\documentclass[12pt]{amsart}
\usepackage{amsfonts,amssymb,amsmath,amscd,amsthm, amstext,bm,color}
\usepackage[colorlinks, citecolor=blue,pagebackref,hypertexnames=false]{hyperref}
\usepackage{enumerate}
\usepackage{graphicx}
\usepackage{epstopdf}
\usepackage{fancyhdr}
\usepackage{geometry}
\usepackage{mathrsfs}
\usepackage{txfonts}
\usepackage{tikz}
\usetikzlibrary{arrows.meta, positioning, decorations.pathreplacing}
\usepackage{url}

\usepackage{comment}
\allowdisplaybreaks
\numberwithin{equation}{section}

\newtheorem{theorem}{Theorem}[section]
\newtheorem{lemma}[theorem]{Lemma}
\newtheorem{proposition}[theorem]{Proposition}
\newtheorem{corollary}[theorem]{Corollary}
\newtheorem{definition}[theorem]{Definition}

\newcommand{\Rp}{\mathbb R_+}

\newcommand{\BMO}{\operatorname{BMO}}

\newcommand{\one}{\mathbf 1}
\newcommand{\eps}{\varepsilon}

\newcommand{\pvs}{\operatorname{p.v.}}

\title[Sharp weighted estimates for Bessel Riesz transforms]{Sharp weighted estimates for the Bessel Riesz transform
and its commutator with Andersen--Kerman weights}

\author[C. Wen]{Chaojie Wen}
\address{ Chaojie Wen,
Department of Mathematics, Sun Yat-sen University, Guangzhou, 510275, P.R.~China.}
\email{wenchj@mail2.sysu.edu.cn}

\date{}

\begin{document}

\begin{abstract}
Let $\lambda>-1/2$, $\lambda\neq0$, and let
$
        \Delta_\lambda=-\frac{d^2}{dx^2}-\frac{2\lambda}{x}\frac{d}{dx}
$
be the Bessel operator on $\Rp=(0,\infty)$ studied by Muckenhoupt and Stein \cite{MS65}. Andersen and Kerman \cite{AK81} proved that for $1<p<\infty$, the Bessel Riesz transform $R_\lambda=\frac{d}{dx}\Delta_\lambda^{-1/2}$ is
bounded on $L^p(\Rp,w(x)\,dx)$ if and only if $w$ is in the intrinsic class $A_{p,\lambda}$. However, the sharp quantitative weighted bound via $[w]_{A_{p,\lambda}}$ was not addressed before.

In this paper, we give a positive answer to this question by proving the sharp quantitative estimate
$$
        \|R_\lambda f\|_{L^p(\Rp,w\,dx)} \le C_{p,\lambda} [w]_{A_{p,\lambda}}^{\max\{1,1/(p-1)\}} \|f\|_{L^p(\Rp,w\,dx)}.
$$
Moreover, for a real-valued function $b$ in the Bessel BMO space $\operatorname{BMO}_\lambda$, the sharp weighted bound for the Riesz commutator is
$$
        \|[b,R_\lambda]f\|_{L^p(\Rp,w\,dx)} \le C_{p,\lambda}\|b\|_{\operatorname{BMO}_\lambda}
        [w]_{A_{p,\lambda}}^{2\cdot\max\{1,1/(p-1)\}} \|f\|_{L^p(\Rp,w\,dx)}.
$$
The argument uses the exact conjugation
$
        U(x)=x^{p-2\lambda-1}w(x), \ d\nu_\lambda=x^{2\lambda+1}\,dx, \ [U]_{A_p(d\nu_\lambda)}=[w]_{A_{p,\lambda}},
$
which reduces the Andersen--Kerman estimate to an $A_p$ weighted estimate for the auxiliary operator
$$
        \mathcal R_\lambda F(x)=\frac{1}{x}R_\lambda(yF(y))(x)
$$
on the space of homogeneous type $(\Rp,|x-y|,d\nu_\lambda)$, whose kernel is a standard Calder\'on--Zygmund
kernel with respect to $\nu_\lambda$. 
\end{abstract}

\maketitle

\section{Introduction}

The Bessel operator
$$
        \Delta_\lambda =-x^{-2\lambda}\frac{d}{dx}x^{2\lambda}\frac{d}{dx} =-\frac{d^2}{dx^2}-\frac{2\lambda}{x}\frac{d}{dx}, \qquad x>0,\ \  \lambda\in(-1/2,\infty)\setminus\{0\},
$$
is a fundamental object in harmonic analysis and complex analysis. Muckenhoupt and Stein \cite{MS65} introduced and developed the associated Poisson
kernels, conjugate functions, fractional integrals, and the Bessel Riesz transforms $$
        R_\lambda=\frac{d}{dx}\Delta_\lambda^{-1/2}.
$$ 
When $2\lambda$ is an integer, $\Delta_\lambda$ coincides with the radial part of the Laplacian in dimension $2\lambda+1$. Consequently, the singular behavior at the endpoint 0 is intrinsic to the operator and does not arise from a particular choice of normalization.

On the Bessel space of homogeneous type $(\Rp,|x-y|,dm_\lambda)$, where $dm_\lambda(x)=x^{2\lambda}\,dx$, the usual
Muckenhoupt classes $A_p(dm_\lambda)$ are the natural classes for the Hardy--Littlewood maximal operator by
Muckenhoupt's theorem \cite{Muckenhoupt1972trans}, and they are the weighted classes used in the standard Bessel
normalization for Calder\'on--Zygmund estimates, see for instance the works of Betancor et al \cite{bfbmt}, and Villani \cite{v08}. Andersen and Kerman \cite{AK81}
considered a different normalization: the same transform acting on the weighted Lebesgue space of the form $L^p(\Rp,w\,dx)$, but not $L^p(\Rp,w\,dm_\lambda)$. The corresponding weights are not the
usual Bessel $A_p$ weights $A_p(dm_\lambda)$. In their normalization, the weight class $A_{p,\lambda}$ is characterized by
\begin{align}\label{eq:AK-characteristic}
        [w]_{A_{p,\lambda}} := \sup_{I\subset\Rp} \left(\frac{1}{\nu_\lambda(I)}\int_I x^p w(x)\,dx\right)
        \left(\frac{1}{\nu_\lambda(I)} \int_I x^{2\lambda p'} w(x)^{-1/(p-1)}\,dx\right)^{p-1}<\infty,
\end{align}
where $d\nu_\lambda(x)=x^{2\lambda+1}\,dx$ and $p'$ is the conjugate exponent of $p$. 

Throughout the paper, $A_{p,\lambda}$ is understood in this sense. In particular, no separate local-finiteness assumption is imposed on
$w\,dx$ at $0$. The defining integrals in \eqref{eq:AK-characteristic} are the relevant finiteness conditions. Thus $A_{p,\lambda}$ 
may contain weights such as $x^{-1-\varepsilon}$. Andersen and Kerman \cite{AK81} proved that, for $1<p<\infty$,
$$
        R_\lambda:L^p(\Rp,w\,dx)\longrightarrow L^p(\Rp,w\,dx)
$$
is bounded if and only if $w\in A_{p,\lambda}$.

We also point out that the difference between  $A_{p,\lambda}$ and  $A_p(dm_\lambda)$  is already visible for power weights. From \eqref{eq:AK-characteristic}, we obtain that
\begin{align}\label{eq:power-AK-intro}
        x^\alpha\in A_{p,\lambda} \quad\Longleftrightarrow\quad -1-p<\alpha<2\lambda p+p-1.
\end{align}
By contrast, converting the usual $A_p(dm_\lambda)$ theorem to the Lebesgue norm requires
$w(x)x^{-2\lambda}\in A_p(dm_\lambda)$ with $w(x)=x^\alpha$ gives
\begin{align}\label{eq:power-usual-converted-intro}
        -1<\alpha<2\lambda+(2\lambda+1)(p-1) =2\lambda p+p-1.
\end{align}
The upper endpoint is the same, while the lower endpoint is not. Hence $A_{p,\lambda}$ contains weights with stronger singularity at the
origin than those obtained from the standard Bessel $A_p$ theorem after the change of measure. We refer to Section \ref{sec:not-usual} for more details on the comparison.

Several weighted and endpoint aspects of Bessel analysis have since been studied, see for example,
\cite{AK81,bcfr,bdt,bfbmt,bfs,bhnv,DLWY,DLMWY,LLLS,LLSW,v08} and the references therein. Some previous closely related results
are the classes $\widetilde A_{p,\lambda}$ introduced by Li--Liang--Lin--Shen \cite{LLLS} and the quantitative
estimates of Li--Liang--Shen--Wick \cite{LLSW}. Their Riesz transform weighted bounds are stated with the shifted class
$\widetilde A_{p,\lambda-\frac12}$. When $-1/2<\lambda<0$, this notation is used below only as a power-measure
characteristic. The comparison involves the measure $x^{2\lambda}\,dx$ and does not require a Bessel-operator
interpretation of the parameter $\lambda-\frac12$. When the weighted norm is written with respect to Lebesgue measure,
$$
        w\in \widetilde A_{p,\lambda-\frac12} \quad\Longleftrightarrow\quad w(x)x^{-2\lambda}\in A_p(dm_\lambda).
$$
Thus $\widetilde A_{p,\lambda-\frac12}$ is the class arising from the usual Bessel $A_p(dm_\lambda)$ theorem after rewriting the norm as
$L^p(\Rp,w\,dx)$, and it is properly contained in $A_{p,\lambda}$. For power weights,
$$
    \begin{array}{rcl}
    x^\alpha\in A_{p,\lambda} &\Longleftrightarrow& -1-p<\alpha<2\lambda p+p-1,\\
    x^\alpha\in \widetilde A_{p,\lambda-\frac12} &\Longleftrightarrow& -1<\alpha<2\lambda p+p-1.
    \end{array}
$$
The unshifted class $\widetilde A_{p,\lambda}$ is not comparable with $A_{p,\lambda}$, since
$$
        x^\alpha\in \widetilde A_{p,\lambda} \quad\Longleftrightarrow\quad -1<\alpha<2\lambda p+2p-1.
$$

The present paper concerns the quantitative dependence on the Andersen--Kerman weights $ A_{p,\lambda}$. Andersen and Kerman identified
the suitable class $A_{p,\lambda}$ for characterizing the weighted $L^p$ boundedness of the Bessel Riesz transform, however, their result does not contain the sharp power of $[w]_{A_{p,\lambda}}$. 

We give a positive answer to this question as follows.
\begin{theorem}\label{thm:intro-riesz}
Let $\lambda>-1/2$, $\lambda\neq0$, and $1<p<\infty$. If $w\in A_{p,\lambda}$, then
\begin{align}\label{eq:intro-sharp-riesz}
        \|R_\lambda f\|_{L^p(\Rp,w\,dx)} \le C_{p,\lambda} [w]_{A_{p,\lambda}}^{\gamma_p}
        \|f\|_{L^p(\Rp,w\,dx)}, \qquad \gamma_p=\max\left\{1,\frac{1}{p-1}\right\}.
\end{align}
In particular,
$$
        \|R_\lambda f\|_{L^2(\Rp,w\,dx)} \le C_\lambda [w]_{A_{2,\lambda}}\|f\|_{L^2(\Rp,w\,dx)}.
$$
The exponent $\gamma_p$ cannot be replaced by a smaller one, see Proposition~\ref{prop:sharpness-powers}.
\end{theorem}

We recall the Bessel Riesz commutator studied by \cite{DLWY}:
$
        [b,R_\lambda]f(x):=b(x)R_\lambda f(x)-R_\lambda(bf)(x),
$
where the symbol $b$ is in the usual Bessel BMO space $\BMO(dm_\lambda)$. 

We now consider symbols $b$ in the BMO space $\BMO_\lambda$, with the norm taken with respect to the measure $\nu_\lambda$:
$$
        \|b\|_{\BMO_\lambda} :=\sup_{I\subset\Rp}\frac{1}{\nu_\lambda(I)} \int_I
        |b-b_{I}^{\nu_\lambda}|\,d\nu_\lambda, \qquad b_{I}^{\nu_\lambda}:=\frac{1}{\nu_\lambda(I)}\int_I b\,d\nu_\lambda.
$$
Our Proposition~\ref{prop:BMO-power-equivalence} gives
$$
        \BMO_\lambda=\BMO(dm_\lambda)
$$
with equivalent norms. Thus the BMO appearing in the next theorem coincides with the usual Bessel BMO space as used in \cite{DLWY}, expressed in the
normalization adapted to the Andersen--Kerman conjugation. 
\begin{theorem}\label{thm:intro-comm}
Let $\lambda>-1/2$, $\lambda\neq0$, $1<p<\infty$, $w\in A_{p,\lambda}$, and let $b\in \BMO_\lambda$ be real-valued. Then
\begin{align}\label{eq:intro-sharp-comm}
        \|[b,R_\lambda]f\|_{L^p(\Rp,w\,dx)} \le C_{p,\lambda}\|b\|_{\BMO_\lambda} [w]_{A_{p,\lambda}}^{2\gamma_p} \|f\|_{L^p(\Rp,w\,dx)}.
\end{align}
In particular,
$$
        \|[b,R_\lambda]f\|_{L^2(\Rp,w\,dx)} \le C_\lambda\|b\|_{\BMO_\lambda} [w]_{A_{2,\lambda}}^2\|f\|_{L^2(\Rp,w\,dx)}.
$$
The exponent $2\gamma_p$ cannot be replaced by a smaller one, see Proposition~\ref{prop:sharpness-powers}.
\end{theorem}

For complex-valued symbols the estimate follows by applying the real-valued case to real and imaginary parts.

The reduction starts from an exact conjugation. Given $w$, set
$
        U(x)=x^{p-2\lambda-1}w(x), \ d\nu_\lambda(x)=x^{2\lambda+1}\,dx.
$
Then
\begin{align}\label{eq:intro-exact-weight}
        [U]_{A_p(d\nu_\lambda)}=[w]_{A_{p,\lambda}}.
\end{align}
If we define $F(x)=\frac{1}{x}f(x)$ and
$$
        \mathcal R_\lambda F(x)=\frac{1}{x}R_\lambda(f)(x),
$$
then
$$
        \|f\|_{L^p(\Rp,w\,dx)}= \|F\|_{L^p(\Rp,U\,d\nu_\lambda)}, \qquad \|R_\lambda f\|_{L^p(\Rp,w\,dx)}= \|\mathcal R_\lambda
        F\|_{L^p(\Rp,U\,d\nu_\lambda)},
$$
and the same identity holds for commutators. The Andersen--Kerman problem is thereby converted into an ordinary
weighted problem for $\mathcal R_\lambda$ on $(\Rp,|x-y|,d\nu_\lambda)$.

With this normalization the Andersen--Kerman kernel estimates become Calder\'on--Zygmund estimates for $\mathcal
R_\lambda$ on $(\Rp,|x-y|,d\nu_\lambda)$. Sparse domination and sharp weighted sparse estimates give
Theorem~\ref{thm:intro-riesz}. The commutator bound follows from applying the Cauchy integral method of Coifman--Rochberg--Weiss after
conjugation. Sharpness is local: on compact subintervals of $\Rp$, the measure $\nu_\lambda$ is comparable with Lebesgue
measure and the Andersen--Kerman kernel has the one-sided sign and size of the Hilbert kernel. The one-sided
Hardy--Hilbert lower examples therefore transfer to the Bessel setting without involving the endpoint $0$.

The remainder is arranged as follows. Section~\ref{sec:geometry} recalls the geometry, the Andersen--Kerman weights, the BMO
normalization, and the exact conjugation. Section~\ref{sec:kernel} converts the Andersen--Kerman kernel estimates into
Calder\'on--Zygmund estimates for $\mathcal R_\lambda$. Section~\ref{sec:max-sparse} contains the
maximal-truncation, sparse-domination, and weighted sparse estimates. Sections~\ref{sec:proof-riesz} and~\ref{sec:proof-comm}
prove the two main inequalities. Section~\ref{sec:not-usual} compares $A_{p,\lambda}$ with the usual Bessel and $\widetilde A$
classes, and Section~\ref{sec:sharpness} proves optimality of the powers.

\section{Geometry, weights, BMO, and the exact conjugation}\label{sec:geometry}

For $x>0$ and $r>0$, write
$$
        B(x,r)=(x-r,x+r)\cap\Rp.
$$
Throughout this paper,
$$
        d\nu_\lambda(x)=x^{2\lambda+1}\,dx.
$$
The measure is distinct from the usual Bessel measure $dm_\lambda=x^{2\lambda}\,dx$ and is
the one produced by the Andersen--Kerman normalization.

\begin{lemma}\label{lem:volume}
For every $x>0$ and $r>0$,
\begin{align}\label{eq:volume}
        \nu_\lambda(B(x,r)) \approx_\lambda r(x+r)^{2\lambda+1}.
\end{align}
Consequently, $(\Rp,|x-y|,d\nu_\lambda)$ is a space of homogeneous type.
\end{lemma}

\begin{proof}

If $0<r\le x/2$, then $t\approx x$ for
$t\in(x-r,x+r)$, and therefore
$$
        \nu_\lambda(B(x,r))
        =\int_{x-r}^{x+r}t^{2\lambda+1}\,dt
        \approx_\lambda
        r x^{2\lambda+1}
        \approx_\lambda
        r(x+r)^{2\lambda+1} .
$$
If $r>x/2$, then $x+r\approx r$.  The upper bound follows from
$$
        \nu_\lambda(B(x,r))
        \le \int_0^{x+r}t^{2\lambda+1}\,dt
        ={(x+r)^{2\lambda+2}\over {2\lambda+2}}
        \lesssim_\lambda r(x+r)^{2\lambda+1} .
$$
For the lower bound, the interval $(x+r/2,x+r)$ is contained in
$B(x,r)$ and has length $r/2$.  Hence
$$
        \nu_\lambda(B(x,r))
        \ge \int_{x+r/2}^{x+r}t^{2\lambda+1}\,dt
        \approx_\lambda r(x+r)^{2\lambda+1} .
$$
This proves \eqref{eq:volume}.

The doubling property follows by replacing $r$ by $2r$.
\end{proof}

\begin{definition}
For $1<p<\infty$, a weight $U$ belongs to $A_p(d\nu_\lambda)$ if
$$
        [U]_{A_p(d\nu_\lambda)} := \sup_I \left(\frac{1}{\nu_\lambda(I)}\int_I U\,d\nu_\lambda\right)
        \left(\frac{1}{\nu_\lambda(I)}\int_I U^{-1/(p-1)}\,d\nu_\lambda\right)^{p-1} <\infty.
$$
\end{definition}

Throughout the paper a weight means a positive measurable function, finite almost everywhere. No local-finiteness condition is imposed on
$w(x)\,dx$ at the origin unless it is part of the displayed characteristic. Under the Andersen--Kerman condition, weights such as
$x^{-1-\varepsilon}$ are allowed. 

\begin{definition}[\cite{AK81}]\label{def:AK-weight}
For $1<p<\infty$, a positive measurable function $w$ belongs to $A_{p,\lambda}$ if
\begin{align}\label{eq:AK-characteristic-main}
        [w]_{A_{p,\lambda}} := \sup_{I\subset\Rp} \left(\frac{1}{\nu_\lambda(I)}\int_I x^p w(x)\,dx\right)
        \left(\frac{1}{\nu_\lambda(I)} \int_I x^{2\lambda p'} w(x)^{-1/(p-1)}\,dx\right)^{p-1}<\infty.
\end{align}
Here and throughout, intervals are finite intervals in $\Rp$. The classical local finiteness condition $w\in L^1_{\rm
loc}(dx)$ at the origin is not imposed. The required assumption is precisely the finiteness of the two defining
integrals in \eqref{eq:AK-characteristic-main}. Thus singular examples such as $w(x)=x^{-1-\varepsilon}$,
$0<\varepsilon<p$, are included whenever they satisfy the displayed condition.
\end{definition}

\begin{definition}
A function $b\in L^1_{\mathrm{loc}}(\Rp,\nu_\lambda)$ belongs to $\BMO_\lambda$ if
$$
        \|b\|_{\BMO_\lambda} := \sup_{I\subset\Rp} \frac{1}{\nu_\lambda(I)} \int_I |b-b_{I}^{\nu_\lambda}|\,d\nu_\lambda<\infty,
$$
where
$$
        b_{I}^{\nu_\lambda}=\frac{1}{\nu_\lambda(I)}\int_I b\,d\nu_\lambda.
$$
\end{definition}

We note that the $\BMO_\lambda$ coincides with the usual Bessel BMO space $\BMO(dm_\lambda)$ as studied in \cite{DLWY}. 

\begin{proposition}\label{prop:BMO-power-equivalence}
Let $a,b>-1$ and put
$
        d\eta_a(x)=x^a\,dx, \ d\eta_b(x)=x^b\,dx \ \text{on }\Rp.
$
For a measurable function $h$ for which the averages below are defined, set
$$
        \|h\|_{\BMO(d\eta_a)} :=\sup_{I\subset\Rp} \frac{1}{\eta_a(I)} \int_I |h-h_I^{\eta_a}|\,d\eta_a, \qquad
        h_I^{\eta_a}:=\frac{1}{\eta_a(I)}\int_I h\,d\eta_a.
$$
Then
$$
        \|h\|_{\BMO(d\eta_a)} \approx_{a,b} \|h\|_{\BMO(d\eta_b)}.
$$
In particular, since $2\lambda>-1$ and $2\lambda+1>-1$,
$$
        \|h\|_{\BMO(dm_\lambda)} \approx_\lambda \|h\|_{\BMO(d\nu_\lambda)}= \|h\|_{\BMO_\lambda}.
$$
\end{proposition}
\begin{proof}
We refer to the known result in \cite{Ho} since both $dm_\lambda=x^{2\lambda}dx$ and $d\nu_\lambda=x^{2\lambda+1}dx$ are doubling measure.
\end{proof}

\begin{proposition}\label{prop:exact-weight}
Let $1<p<\infty$ and define
\begin{align}\label{eq:U-definition}
        U(x):=x^{p-2\lambda-1}w(x).
\end{align}
Then
\begin{align}\label{eq:exact-characteristic}
        [U]_{A_p(d\nu_\lambda)}=[w]_{A_{p,\lambda}}.
\end{align}
In particular,
$$
        w\in A_{p,\lambda} \quad\Longleftrightarrow\quad U\in A_p(d\nu_\lambda).
$$
\end{proposition}

\begin{proof}
For every interval $I\subset\Rp$,
$$
        \int_I U\,d\nu_\lambda =\int_I x^{p-2\lambda-1}w(x)x^{2\lambda+1}\,dx =\int_I x^p w(x)\,dx.
$$
Also,
$$
        \int_I U^{-1/(p-1)}\,d\nu_\lambda =\int_I x^{-(p-2\lambda-1)/(p-1)}w(x)^{-1/(p-1)}
        x^{2\lambda+1}\,dx =\int_I x^{2\lambda p'}w(x)^{-1/(p-1)}\,dx,
$$
because
$$
        2\lambda+1-\frac{p-2\lambda-1}{p-1}=2\lambda p'.
$$
Putting these two identities into the definition of $A_p(d\nu_\lambda)$ gives \eqref{eq:exact-characteristic}.
\end{proof}

Recall the conjugated operator
\begin{align}\label{eq:conjugated-operator}
        \mathcal R_\lambda F(x):=\frac{1}{x}R_\lambda (f)(x),\quad \text{where }f(y)=yF(y).
\end{align}
Under this conjugation the Andersen--Kerman weight becomes an ordinary $A_p(d\nu_\lambda)$ weight.

\begin{proposition}\label{prop:norm-identities}
Let $1<p<\infty$, let $U$ be as in \eqref{eq:U-definition}, and put $f(x)=xF(x)$. Then
\begin{align}\label{eq:norm-F}
        \|f\|_{L^p(\Rp,w\,dx)}= \|F\|_{L^p(\Rp,U\,d\nu_\lambda)}
\end{align}
and
\begin{align}\label{eq:norm-RF}
        \|R_\lambda f\|_{L^p(\Rp,w\,dx)}= \|\mathcal R_\lambda F\|_{L^p(\Rp,U\,d\nu_\lambda)}.
\end{align}
Moreover,
\begin{align}\label{eq:norm-comm}
        \|[b,R_\lambda]f\|_{L^p(\Rp,w\,dx)}= \|[b,\mathcal R_\lambda]F\|_{L^p(\Rp,U\,d\nu_\lambda)}.
\end{align}
\end{proposition}

\begin{proof}
Since $U\,d\nu_\lambda=x^p w(x)\,dx$,
$$
        \|F\|_{L^p(\Rp,U\,d\nu_\lambda)}^p =\int_0^\infty |F(x)|^p x^p w(x)\,dx =\int_0^\infty |f(x)|^p w(x)\,dx.
$$
This gives \eqref{eq:norm-F}. Since $\mathcal R_\lambda F=x^{-1}R_\lambda f$,
$$
        \|\mathcal R_\lambda F\|_{L^p(\Rp,U\,d\nu_\lambda)}^p =\int_0^\infty \left|\frac{R_\lambda
        f(x)}{x}\right|^p x^p w(x)\,dx =\int_0^\infty |R_\lambda f(x)|^p w(x)\,dx.
$$
This gives \eqref{eq:norm-RF}. Finally,
$$
        [b,\mathcal R_\lambda]F(x) =b(x)\frac{1}{x}R_\lambda(yF(y))(x) -\frac{1}{x}R_\lambda(yb(y)F(y))(x) =\frac{1}{x}[b,R_\lambda]f(x).
$$
The same norm calculation proves \eqref{eq:norm-comm}.
\end{proof}

\section{Kernel estimates for the conjugated operator}\label{sec:kernel}

Let $K_\lambda(x,y)$ denote the kernel of $R_\lambda$ with respect to $dm_\lambda(y)=y^{2\lambda}\,dy$:
\begin{align}\label{eq:R-kernel-dm}
        R_\lambda f(x)=\pvs\int_0^\infty K_\lambda(x,y)f(y)\,dm_\lambda(y).
\end{align}
Then $\mathcal R_\lambda$ has kernel with respect to $d\nu_\lambda(y)=y^{2\lambda+1}\,dy$ given by
\begin{align}\label{eq:conjugated-kernel}
        \mathcal K_\lambda(x,y)=\frac{1}{x}K_\lambda(x,y).
\end{align}
If $f(y)=yF(y)$, then
$$
        \mathcal R_\lambda F(x) =\frac{1}{x}\pvs\int_0^\infty K_\lambda(x,y)yF(y)y^{2\lambda}\,dy
        =\pvs\int_0^\infty \frac{K_\lambda(x,y)}{x}F(y)y^{2\lambda+1}\,dy.
$$

The Andersen--Kerman estimates are used below in the following form, stated directly for the conjugated kernel. In the far region
$y\ge2x$, differentiating $K_\lambda(x,y)/x$ requires the cancellation in $x\partial_x K_\lambda(x,y)-K_\lambda(x,y)$, so separate
estimates for $K_\lambda$ and $\partial_x K_\lambda$ are insufficient.

The statement also records the local information used for sharpness. After subtraction of the Hilbert principal part, the remainder is
logarithmic in general. The lower-bound argument uses only the one-sided sign and size of the principal singularity.

Before stating the following lemma, we declare that for the rest of the paper we use the notation $J\Subset\Rp$ to mean that $J$ is a compact interval contained in $(0,\infty)$.

\begin{lemma}\label{lem:AK-kernel-package}
There is a constant $C_\lambda$ such that $\mathcal K_\lambda$ is $C^1$ away from the diagonal and the following estimates hold.
If $x/2<y<2x$ and $x\ne y$, then
\begin{align}\label{eq:AK-local-conj-size}
        |\mathcal K_\lambda(x,y)| \le C_\lambda\frac{1}{x^{2\lambda+1}|x-y|},
\end{align}
\begin{align}\label{eq:AK-local-conj-deriv}
        |\partial_x\mathcal K_\lambda(x,y)|+|\partial_y\mathcal K_\lambda(x,y)| \le C_\lambda\frac{1}{x^{2\lambda+1}|x-y|^2}.
\end{align}
Moreover, for every compact interval $J\Subset\Rp$ there are a non-zero constant $c_\lambda$
and a remainder $E_{\lambda,J}$ such that, for $x,y\in J$, $x\ne y$,
\begin{align}\label{eq:AK-local-leading-term}
        \mathcal K_\lambda(x,y) =\frac{c_\lambda}{x^{2\lambda+1}}\frac{1}{x-y}+E_{\lambda,J}(x,y),
\end{align}
and
\begin{align}\label{eq:AK-local-log-remainder}
        |E_{\lambda,J}(x,y)| \le C_{\lambda,J} \left(1+\log^+\frac{\ell(J)}{|x-y|}\right).
\end{align}
Here $\ell(J)$ denotes the length of $J$. In addition, there are $\varepsilon_\lambda\in\{-1,1\}$, $s_J>1$, and
$c_{\lambda,J}>0$ such that, whenever $x,y\in J$, $x\ne y$, and $s_J^{-1}<y/x<s_J$,
\begin{align}\label{eq:AK-local-sign-conj}
        \varepsilon_\lambda \,\operatorname{sgn}(x-y)\, \mathcal K_\lambda(x,y)y^{2\lambda+1} \ge \frac{c_{\lambda,J}}{|x-y|}.
\end{align}

If $y\ge2x$, then
\begin{align}\label{eq:AK-right-conj-size}
        |\mathcal K_\lambda(x,y)| \le C_\lambda\frac{1}{y^{2\lambda+2}},
\end{align}
\begin{align}\label{eq:AK-right-conj-deriv}
        |\partial_x\mathcal K_\lambda(x,y)|+|\partial_y\mathcal K_\lambda(x,y)| \le C_\lambda\frac{1}{y^{2\lambda+3}}.
\end{align}

If $x\ge2y$, then
\begin{align}\label{eq:AK-left-conj-size}
        |\mathcal K_\lambda(x,y)| \le C_\lambda\frac{1}{x^{2\lambda+2}},
\end{align}
\begin{align}\label{eq:AK-left-conj-deriv}
        |\partial_x\mathcal K_\lambda(x,y)|+|\partial_y\mathcal K_\lambda(x,y)| \le C_\lambda\frac{1}{x^{2\lambda+3}}.
\end{align}
\end{lemma}

\begin{proof}
The normalization is as follows. With the normalization of $R_\lambda$ used in \eqref{eq:R-kernel-dm}, Andersen and Kerman express the generalized Hankel conjugate transform, for $\tau>0$, by the Poisson-regularized kernel $Q_\lambda(\tau,x,y)$ in the notation of \cite[(1.2)--(1.3)]{AK81}:
$$
        Q_\lambda(\tau,x,y) =-(xy)^{-\lambda+1/2} \int_0^\infty e^{-\tau r} J_{\lambda+1/2}(xr)J_{\lambda-1/2}(yr)\,r\,dr.
$$
The principal-value kernel $K_\lambda(x,y)$ in \eqref{eq:R-kernel-dm} is obtained as the Abel limit
$$
        K_\lambda(x,y)=\lim_{\tau\downarrow0}Q_\lambda(\tau,x,y),
        \qquad x\ne y.
$$

For $-1/2<\lambda\le0$, the same kernel is obtained from the compensated heat-potential definition of Betancor--Harboure--Nowak--Viviani \cite[pp.~104--105]{bhnv}. The Hankel-transform identity given there is valid for every $\lambda>-1/2$ and agrees with the Bessel-product normalization above.

For $\tau>0$, the change of variables $\rho=xr$ gives
$$
        Q_\lambda(\tau,x,y)
        =x^{-2\lambda-1}Q_\lambda(\tau/x,1,y/x).
$$
Taking the Abel limit and setting
$\Psi_\lambda(t)=\lim_{s\downarrow0}Q_\lambda(s,1,t)$ yields the homogeneous representation
\begin{align}\label{eq:AK-Phi-representation}
        K_\lambda(x,y)=x^{-2\lambda-1}\Psi_\lambda(y/x), \qquad \mathcal K_\lambda(x,y)=x^{-2\lambda-2}\Phi_\lambda(y/x),
\end{align}
for $x,y>0$, $x\ne y$, where $\Phi_\lambda$ is the corresponding function for the conjugated kernel. The Calder\'on--Zygmund size and smoothness estimites therefore retain the same form under simultaneous dilations of $x$ and $y$.

For the full range $\lambda>-1/2$, we use the heat-kernel estimates of Betancor--Harboure--Nowak--Viviani \cite[pp.~117--119]{bhnv}. In the present notation, they give
\begin{align}\label{eq:AK-pointwise-far-original}
 |K_\lambda(x,y)|&\le C_\lambda x^{-2\lambda-1}, &&0<y\le x/2,\\
 |K_\lambda(x,y)|&\le C_\lambda x y^{-2\lambda-2}, &&y\ge2x,
\end{align}
and, when $x/2<y<2x$,
\begin{align}\label{eq:AK-pointwise-local-original}
 K_\lambda(x,y)
 =-\frac{1}{\pi}(xy)^{-\lambda}\frac{1}{x-y}
 +O_\lambda\!\left(y^{-2\lambda-1}
 \left(1+\log^+\frac{xy}{(x-y)^2}\right)\right).
\end{align}
After division by $x$ and the substitution $t=y/x$, this becomes
\begin{align}\label{eq:AK-Phi-principal}
 \Phi_\lambda(t)
 =-\frac{1}{\pi}\frac{t^{-\lambda}}{1-t}
 +O_\lambda\!\left(t^{-2\lambda-1}
 \left(1+\log^+\frac{t}{(1-t)^2}\right)\right),
 \qquad \frac12<t<2,
\end{align}
so that
\begin{align}\label{eq:AK-Phi-local}
 |\Phi_\lambda(t)|\le \frac{C_\lambda}{|1-t|},
 \qquad \frac12<t<2,\quad t\ne1.
\end{align}

We next obtain the derivative estimate without using the classical angular formula, which is available only for $\lambda>0$. By the gradient criterion and the Riesz-kernel theorem of Castro--Szarek \cite[(2.4), p.~639, Theorem~2.2, p.~641, and the proof of the Riesz-transform case, p.~647]{CS14}, the first-order Bessel Riesz kernel satisfies, for every $\lambda>-1/2$,
\begin{align*}
 |\partial_xK_\lambda(x,y)|+|\partial_yK_\lambda(x,y)|
 \le \frac{C_\lambda}{|x-y|\,m_\lambda(B(x,|x-y|))}.
\end{align*}
If $x/2<y<2x$ and $r=|x-y|$, then $r<x$ and
$m_\lambda(B(x,r))\simeq_\lambda r x^{2\lambda}$. Together with \eqref{eq:AK-pointwise-local-original}, this gives
\begin{align*}
 |\partial_x\mathcal K_\lambda(x,y)|+|\partial_y\mathcal K_\lambda(x,y)|
 &\le \frac{1}{x}\bigl(|\partial_xK_\lambda(x,y)|+|\partial_yK_\lambda(x,y)|\bigr)
       +\frac{|K_\lambda(x,y)|}{x^2}
 \le C_\lambda\frac{1}{x^{2\lambda+1}r^2}.
\end{align*}
Since $\partial_y\mathcal K_\lambda(x,y)=x^{-2\lambda-3}\Phi_\lambda'(y/x)$, it follows that
\begin{align}\label{eq:AK-Phi-local-derivative}
 |\Phi_\lambda'(t)|\le \frac{C_\lambda}{|1-t|^2},
 \qquad \frac12<t<2,\quad t\ne1.
\end{align}

The local sign estimate follows directly from \eqref{eq:AK-Phi-principal}. Since the principal term has size $|1-t|^{-1}$ and the remainder is logarithmic, we can choose $s>1$ close enough to $1$ so that the principal term dominates for $s^{-1}<t<s$, $t\ne1$. Thus, there exist $\varepsilon_\lambda\in\{-1,1\}$ and $c_\lambda'>0$ such that
\begin{align}\label{eq:AK-Phi-sign}
        \varepsilon_\lambda\operatorname{sgn}(1-t) \Phi_\lambda(t)t^{2\lambda+1} \ge \frac{c_\lambda'}{|1-t|}, \qquad
        s^{-1}<t<s, \quad t\ne1.
\end{align}

For the far regions, use the homogeneity
$\mathcal K_\lambda(rx,ry)=r^{-2\lambda-2}\mathcal K_\lambda(x,y)$ and put
\begin{align}\label{eq:AK-far}
 A_\lambda(t):=\mathcal K_\lambda(t,1),\qquad
 B_\lambda(t):=\mathcal K_\lambda(1,t),\qquad 0<t\le\frac12.
\end{align}
The estimates in \eqref{eq:AK-pointwise-far-original} show that $A_\lambda$ and $B_\lambda$ are bounded. To control their derivatives at $t=0$, we use the explicit hypergeometric representation for $K_\lambda$ given by Betancor--Harboure--Nowak--Viviani \cite[p.~117]{bhnv}. Its hypergeometric argument is
$4t^2/(1+t^2)^2$, which stays in $[0,16/25]$ for $0\le t\le1/2$. After the powers of $t$ are factored out, that formula shows that there are functions $a_\lambda,b_\lambda\in C^\infty([0,1/4])$ such that
\begin{align*}
 K_\lambda(t,1)=t\,a_\lambda(t^2),\qquad
 K_\lambda(1,t)=b_\lambda(t^2),\qquad 0<t\le\frac12.
\end{align*}
Therefore $A_\lambda(t)=a_\lambda(t^2)$ and $B_\lambda(t)=b_\lambda(t^2)$, and both functions extend to $C^1$ functions on $[0,1/2]$. Hence
\begin{align}\label{eq:AK-far-bounds}
 |A_\lambda(t)|+|B_\lambda(t)|+|A_\lambda'(t)|+|B_\lambda'(t)|
 \le C_\lambda,\qquad 0<t\le\frac12.
\end{align}

These single-variable estimates now immediately translate back to the announced Calder\'on--Zygmund bounds. In the local region, setting $t=y/x$ in \eqref{eq:AK-Phi-representation} gives $\mathcal K_\lambda(x,y)=x^{-2\lambda-2}\Phi_\lambda(t)$ and $|x-y|=x|1-t|$. The size estimate \eqref{eq:AK-Phi-local} implies \eqref{eq:AK-local-conj-size}. Differentiating the identity yields
\begin{align}
        \partial_y\mathcal K_\lambda(x,y) &=x^{-2\lambda-3}\Phi_\lambda'(t), \notag \\
        \partial_x\mathcal K_\lambda(x,y) &=x^{-2\lambda-3}\bigl[-(2\lambda+2)\Phi_\lambda(t)-t\Phi_\lambda'(t)\bigr]. \notag
\end{align}
Since $1/2<t<2$, combining \eqref{eq:AK-Phi-local} and \eqref{eq:AK-Phi-local-derivative} produces the local derivative bound \eqref{eq:AK-local-conj-deriv}. The principal expansion \eqref{eq:AK-Phi-principal} implies that for $x,y\in J\Subset\Rp$ satisfying $x/2<y<2x$,
$$
        \mathcal K_\lambda(x,y) =x^{-2\lambda-2}\frac{c_\lambda (y/x)^{-\lambda}}{1-y/x}
        +x^{-2\lambda-2} O_\lambda\left(1+\log^+\frac{x}{|x-y|}\right).
$$
Because $(y/x)^{-\lambda}=1+O_\lambda(|1-y/x|)$ within the local cone, the difference between this principal term and $c_\lambda x^{-2\lambda-1}(x-y)^{-1}$ is bounded on the part of $J\times J$ inside the local cone. Since $x$ is bounded above and below on the fixed compact interval $J$, the required estimates hold there.
On the part of $J\times J$ outside the local cone $x/2<y<2x$, the points are separated from the diagonal. Hence the same remainder bound holds there, after increasing $C_{\lambda,J}$, by the continuity of the kernel away from the diagonal.
Thus properties \eqref{eq:AK-local-leading-term}--\eqref{eq:AK-local-log-remainder} hold. The local sign condition \eqref{eq:AK-local-sign-conj} follows by rewriting \eqref{eq:AK-Phi-sign} as
$$
        \varepsilon_\lambda\operatorname{sgn}(x-y) \mathcal K_\lambda(x,y)y^{2\lambda+1} =\varepsilon_\lambda\operatorname{sgn}(1-t)
        x^{-1}\Phi_\lambda(t)t^{2\lambda+1} \ge \frac{c}{x|1-t|} =\frac{c}{|x-y|}.
$$

Finally, if $y\ge2x$, then $\mathcal K_\lambda(x,y)=y^{-2\lambda-2}A_\lambda(x/y)$. The boundedness of $A_\lambda$ and $A_\lambda'$ from \eqref{eq:AK-far-bounds} immediately gives the right-far bounds \eqref{eq:AK-right-conj-size} and \eqref{eq:AK-right-conj-deriv}. Symmetrically, if $x\ge2y$, then $\mathcal K_\lambda(x,y)=x^{-2\lambda-2}B_\lambda(y/x)$, which yields the left-far bounds \eqref{eq:AK-left-conj-size} and \eqref{eq:AK-left-conj-deriv}. The proof is complete.
\end{proof}

\begin{proposition}\label{prop:AK-CZ-kernel}
There are constants $C_\lambda>0$ and $\delta\in(0,1]$, depending only on $\lambda$, such that for all $x\neq y$,
\begin{align}\label{eq:CZ-size}
        |\mathcal K_\lambda(x,y)| \le \frac{C_\lambda}{\nu_\lambda(B(x,|x-y|))},
\end{align}
and, whenever $|x-x'|\le |x-y|/2$,
\begin{align}\label{eq:CZ-smooth-x}
        |\mathcal K_\lambda(x,y)-\mathcal K_\lambda(x',y)| \le C_\lambda \left(\frac{|x-x'|}{|x-y|}\right)^\delta
        \frac{1}{\nu_\lambda(B(x,|x-y|))}.
\end{align}
Similarly, whenever $|y-y'|\le |x-y|/2$,
\begin{align}\label{eq:CZ-smooth-y}
        |\mathcal K_\lambda(x,y)-\mathcal K_\lambda(x,y')| \le C_\lambda \left(\frac{|y-y'|}{|x-y|}\right)^\delta
        \frac{1}{\nu_\lambda(B(x,|x-y|))}.
\end{align}
Moreover, $\mathcal R_\lambda$ is bounded on $L^2(\Rp, d\nu_\lambda)$.
\end{proposition}

\begin{proof}
By Lemma~\ref{lem:volume}, for every $x>0$ and $r>0$,
\begin{align}\label{eq:V-r-xr-kernel-section}
        \nu_\lambda(B(x,r))\approx_\lambda r(x+r)^{2\lambda+1}.
\end{align}
Let $r=|x-y|$. If $x/2<y<2x$, then $\nu_\lambda(B(x,r))\approx_\lambda rx^{2\lambda+1}$, and
\eqref{eq:AK-local-conj-size} gives $|\mathcal K_\lambda(x,y)|\lesssim_\lambda \nu_\lambda(B(x,r))^{-1}$. If $y\ge2x$, then $r\approx y$
and $\nu_\lambda(B(x,r))\approx_\lambda y^{2\lambda+2}$, and \eqref{eq:AK-right-conj-size} gives the same conclusion. If $x\ge2y$, then $r\approx x$ and
$\nu_\lambda(B(x,r))\approx_\lambda x^{2\lambda+2}$, and \eqref{eq:AK-left-conj-size} gives \eqref{eq:CZ-size}.

The derivative estimates in Lemma~\ref{lem:AK-kernel-package} imply
\begin{align}\label{eq:global-derivative-Kc}
        |\partial_x\mathcal K_\lambda(x,y)|+|\partial_y\mathcal K_\lambda(x,y)| \le
        \frac{C_\lambda}{r\,\nu_\lambda(B(x,r))} , \qquad x\ne y.
\end{align}
Indeed, the three regions give respectively $x^{-2\lambda-1}r^{-2}$, $y^{-2\lambda-3}$, and $x^{-2\lambda-3}$, which are the right side
of \eqref{eq:global-derivative-Kc} by \eqref{eq:V-r-xr-kernel-section}.

Assume $|x-x'|\le r/2$. For $x_\theta=x+\theta(x'-x)$ and $r_\theta=|x_\theta-y|$, one has $\frac{r}{2}\le r_\theta\le\frac{3r}{2}$ and
\begin{equation}\label{eq:volume-stability-x}
        \nu_\lambda(B(x_\theta,r_\theta))
        \approx_\lambda
        \nu_\lambda(B(x,r)).
\end{equation}
Indeed, by \eqref{eq:V-r-xr-kernel-section},
$$
        \nu_\lambda(B(x_\theta,r_\theta))
        \approx_\lambda
        r_\theta(x_\theta+r_\theta)^{2\lambda+1}.
$$
We already know $r_\theta\approx r$.  Now we show
$
        x_\theta+r_\theta \approx x+r.
$

To see this, first note that
$$
        x_\theta+r_\theta
        \le x+{r\over2}+{3r\over2}
        =x+2r
        \le 2(x+r).
$$
For the lower bound, if $r\le x$, then
$x_\theta\ge x-r/2\ge x/2$, so
$$
        x_\theta+r_\theta\ge x/2\ge {1\over4}(x+r).
$$
If $r>x$, then $r_\theta\ge r/2$ and $x+r<2r$, so
$$
        x_\theta+r_\theta\ge r/2>{1\over4}(x+r).
$$
Thus $x_\theta+r_\theta\approx x+r$, and
\eqref{eq:volume-stability-x} follows by combining \eqref{eq:V-r-xr-kernel-section} and $r_\theta \approx r$.

Then, the mean value theorem and \eqref{eq:global-derivative-Kc} give
\begin{align*}
        |\mathcal K_\lambda(x,y)-\mathcal K_\lambda(x',y)| 
        &\le |x-x'|\sup_{0\le\theta\le1}|\partial_x\mathcal
        K_\lambda(x_\theta,y)| \\
        &\le C_\lambda\frac{|x-x'|}{r}\sup_{0\le\theta\le1}\frac{1}{\nu_\lambda(B(x_\theta,r_\theta))}
        \le C_\lambda\frac{|x-x'|}{r}\frac{1}{\nu_\lambda(B(x,r))}. \notag
\end{align*}
Thus \eqref{eq:CZ-smooth-x} holds with $\delta=1$. The proof of \eqref{eq:CZ-smooth-y} is
identical, with the center $x$ fixed and $y_\theta=y+\theta(y'-y)$.

For the $L^2(\Rp,d\nu_\lambda)$ normalization, Andersen--Kerman's theorem gives $L^2(\Rp,w\,dx)$ boundedness of $R_\lambda$ for every $w\in
A_{2,\lambda}$ in the sense of \eqref{eq:AK-characteristic-main}. Take
$$
        w_0(x)=x^{2\lambda-1}.
$$
The weight $w_0$ is admissible in the Andersen--Kerman sense even when $w_0\,dx$ is not locally finite at the origin, because
$$
        x^2w_0(x)=x^{2\lambda+1}, \qquad x^{4\lambda}w_0(x)^{-1}=x^{2\lambda+1}.
$$
Consequently $[w_0]_{A_{2,\lambda}}=1$. If $f=xF$, then
$$
        \|f\|_{L^2(\Rp,w_0\,dx)}=\|F\|_{L^2(\Rp,d\nu_\lambda)}, \qquad \|R_\lambda f\|_{L^2(\Rp,w_0\,dx)}= \|\mathcal R_\lambda F\|_{L^2(\Rp,d\nu_\lambda)}.
$$
Applying Andersen--Kerman's $p=2$ theorem with $w_0$ proves the $L^2(\Rp,d\nu_\lambda)$ boundedness of $\mathcal R_\lambda$.
\end{proof}

\section{Maximal functions and sparse domination}\label{sec:max-sparse}

\subsection{Maximal truncations}\label{sec:maxtrunc}

For $\eps>0$, define
$$
        \mathcal R_{\lambda,\eps}F(x) =\int_{|x-y|>\eps}\mathcal K_\lambda(x,y)F(y)\,d\nu_\lambda(y),
$$
and
$$
        \mathcal R_{\lambda,*}F(x) =\sup_{\eps>0}|\mathcal R_{\lambda,\eps}F(x)|.
$$
Let $M_\lambda$ denote the Hardy--Littlewood maximal operator associated to $\nu_\lambda$:
$$
        M_\lambda F(x) =\sup_{I\ni x,\, I\subset\Rp}\frac{1}{\nu_\lambda(I)}\int_I |F(y)|\,d\nu_\lambda(y).
$$

\begin{lemma}\label{lem:M-weak}
For every $F\in L^1(\Rp, d\nu_\lambda)$ and every $\alpha>0$,
$$
        \nu_\lambda(\{x\in\Rp:M_\lambda F(x)>\alpha\}) \le \frac{C_\lambda}{\alpha}\|F\|_{L^1(\Rp, d\nu_\lambda)}.
$$
\end{lemma}

\begin{proof}
The covering argument for doubling intervals applies. Let $E_\alpha=\{x:M_\lambda F(x)>\alpha\}$.
For each $x\in E_\alpha$ choose an interval $I_x\ni x$ with
$$
        \int_{I_x}|F|\,d\nu_\lambda>\alpha\nu_\lambda(I_x).
$$
By the one-dimensional Vitali covering lemma, there is a pairwise disjoint subcollection $\{I_j\}$ such
that $E_\alpha\subset\bigcup_j 5I_j$. Since $\nu_\lambda$ is doubling,
$$
        \nu_\lambda(E_\alpha) \le\sum_j\nu_\lambda(5I_j) \lesssim_\lambda\sum_j\nu_\lambda(I_j) \le
        \frac{1}{\alpha}\sum_j\int_{I_j}|F|\,d\nu_\lambda \le \frac{1}{\alpha}\|F\|_{L^1(\Rp, d\nu_\lambda)}.
$$
The proof is complete.
\end{proof}

\begin{lemma}\label{lem:Tstar-weak}
There is a constant $C_\lambda$ such that for every bounded compactly supported $F$ and every $\alpha>0$,
\begin{align}\label{eq:Tstar-weak}
        \nu_\lambda(\{x\in\Rp:\mathcal R_{\lambda,*}F(x)>\alpha\}) \le \frac{C_\lambda}{\alpha}\|F\|_{L^1(\Rp, d\nu_\lambda)}.
\end{align}
Moreover $\mathcal R_{\lambda,*}$ is bounded on $L^2(\Rp, d\nu_\lambda)$.
\end{lemma}

\begin{proof}
By Proposition~\ref{prop:AK-CZ-kernel}, $\mathcal R_\lambda$ is an $L^2(\Rp, d\nu_\lambda)$-bounded Calder\'on--Zygmund operator
on the space of homogeneous type $(\Rp,|x-y|,d\nu_\lambda)$. The maximal-truncation theorem for Calder\'on--Zygmund
operators on spaces of homogeneous type gives both the $L^2$ boundedness of $\mathcal R_{\lambda,*}$ and the weak $(1,1)$
estimate \eqref{eq:Tstar-weak}, see for instance, Coifman--Weiss \cite{CW77}.
\end{proof}

\subsection{Sparse domination for the conjugated transform}\label{sec:sparse}

\begin{definition}
A collection $\mathcal S$ of intervals in $\Rp$ is sparse if for each $I\in\mathcal S$ there
exists a measurable set $E_I\subset I$ such that
$$
        \nu_\lambda(E_I)\ge \frac{1}{2}\nu_\lambda(I)
$$
and the sets $\{E_I:I\in\mathcal S\}$ are pairwise disjoint.
\end{definition}

The required form of sparse domination is the following bilinear estimate for Calder\'on--Zygmund operators on spaces of homogeneous type.

\begin{theorem}\label{thm:sparse-domination}
Let $F,G$ be bounded functions with compact support in $\Rp$. Then there is a sparse collection $\mathcal S$ of intervals such that
\begin{align}\label{eq:sparse-domination}
        \left| \int_{\Rp}\mathcal R_\lambda F(x)G(x)\,d\nu_\lambda(x) \right| \le C_\lambda
        \sum_{I\in\mathcal S} \langle |F|\rangle_I\langle |G|\rangle_I\nu_\lambda(I).
\end{align}
\end{theorem}

\begin{proof}
By Proposition~\ref{prop:AK-CZ-kernel}, $\mathcal R_\lambda$ is an $L^2(\Rp, d\nu_\lambda)$-bounded Calder\'on--Zygmund operator on the
space of homogeneous type $(\Rp,|x-y|,d\nu_\lambda)$. Its kernel satisfies the size and H\"older regularity estimates
\eqref{eq:CZ-size}--\eqref{eq:CZ-smooth-y}, and Lemma~\ref{lem:Tstar-weak} gives the weak $(1,1)$ bound for the maximal truncation.
The sparse domination theorem for Calder\'on--Zygmund operators on spaces of homogeneous type then yields
\eqref{eq:sparse-domination}, see Lerner \cite{Ler13} for the Euclidean principle and Lorist \cite{Lorist21} for a
space-of-homogeneous-type formulation. The constants depend only on the doubling constant of $\nu_\lambda$, the kernel constants, and
the $L^2(\Rp, d\nu_\lambda)$ operator norm of $\mathcal R_\lambda$, in particular they depend only on $\lambda$.
\end{proof}

\subsection{Sharp weighted estimates for sparse forms}\label{sec:sparse-weighted}

The weighted estimate needed below is the following positive sparse estimate.

\begin{theorem}\label{thm:weighted-sparse}
Let $1<p<\infty$, let $U\in A_p(d\nu_\lambda)$, and put
$
        \sigma=U^{-1/(p-1)}.
$
If $\mathcal S$ is sparse, then, for all bounded compactly supported $F,G$,
\begin{align}\label{eq:weighted-sparse}
        \sum_{I\in\mathcal S} \langle |F|\rangle_I\langle |G|\rangle_I\nu_\lambda(I) &\le
        C_{p,\lambda}[U]_{A_p(d\nu_\lambda)}^{\gamma_p} \|F\|_{L^p(\Rp,U\,d\nu_\lambda)} \|G\|_{L^{p'}(\Rp,\sigma\,d\nu_\lambda)},
\end{align}
where
$$
        \gamma_p=\max\left\{1,\frac{1}{p-1}\right\}.
$$
\end{theorem}

\begin{proof}
Let $M_\lambda$ denote the Hardy--Littlewood maximal operator relative to
$\nu_\lambda$. Let $[V]_{A_\infty(d\nu_\lambda)}$ be the Fujii--Wilson characteristic
\[
        [V]_{A_\infty(d\nu_\lambda)}
        :=\sup_I\frac{1}{V(I)}\int_I M_\lambda(V\one_I)\,d\nu_\lambda .
\]
We use the mixed $A_p$--$A_\infty$ estimate for sparse bilinear forms,
\cite[Theorem~1.6]{Li17}. Applied with $p_0=1$, $q_0=\infty$, $p=q$,
$w=U$, and $\sigma=U^{-1/(p-1)}$, this theorem gives

\begin{align}\label{eq:mixed-sparse-bound}
        \sum_{I\in\mathcal S} \langle |F|\rangle_I\langle |G|\rangle_I\nu_\lambda(I) &\le C_{p,\lambda}[U]_{A_p(d\nu_\lambda)}^{1/p}
        \bigl([U]_{A_\infty(d\nu_\lambda)}^{1/p'} +[\sigma]_{A_\infty(d\nu_\lambda)}^{1/p}\bigr) \ \|F\|_{L^p(\Rp,U\,d\nu_\lambda)} \|G\|_{L^{p'}(\Rp,\sigma\,d\nu_\lambda)}.
\end{align}

Indeed, in the notation of \cite[Theorem~1.6]{Li17}, the auxiliary weights are
$u=U^{1/(1-p)}=\sigma$, $v=\sigma^{1/(1-p')}=U$, and the corresponding
Muckenhoupt exponent is $r=p$.
The standard comparisons for the Fujii--Wilson characteristic on spaces of
homogeneous type give
$$
        [U]_{A_\infty(d\nu_\lambda)}\lesssim_{p,\lambda} [U]_{A_p(d\nu_\lambda)}, \qquad
        [\sigma]_{A_\infty(d\nu_\lambda)}\lesssim_{p,\lambda} [\sigma]_{A_{p'}(d\nu_\lambda)} =[U]_{A_p(d\nu_\lambda)}^{1/(p-1)}.
$$
Substituting these two estimates into \eqref{eq:mixed-sparse-bound} gives the powers
$$
        \frac1p+\frac1{p'}=1, \qquad \frac1p+\frac1{p(p-1)}=\frac1{p-1}.
$$
Thus the exponent is $\max\{1,1/(p-1)\}=\gamma_p$, which proves \eqref{eq:weighted-sparse}.
\end{proof}

For commutators, the sharp one-weight power is obtained by the Cauchy integral method \cite{CRW76}. The required perturbation lemma is stated next.

\begin{lemma}\label{lem:exp-perturb}
Let $U\in A_p(d\nu_\lambda)$ and let $b\in\BMO_\lambda$ be real-valued. There are constants
$c_{p,\lambda},C_{p,\lambda}>0$ such that the following holds: if $\|b\|_{\BMO_\lambda}>0$ and
$$
        |z|\le \frac{c_{p,\lambda}}{\|b\|_{\BMO_\lambda}[U]_{A_p(d\nu_\lambda)}^{\gamma_p}},
$$
then the weight
$
        U_z=e^{p\operatorname{Re}z\,b}U
$
belongs to $A_p(d\nu_\lambda)$ and
\begin{align}\label{eq:perturbed-Ap}
        [U_z]_{A_p(d\nu_\lambda)} \le C_{p,\lambda}[U]_{A_p(d\nu_\lambda)}.
\end{align}
If $\|b\|_{\BMO_\lambda}=0$, then \eqref{eq:perturbed-Ap} also holds for all $z\in\mathbb C$.
\end{lemma}

\begin{proof}
Set
$
        \sigma=U^{-1/(p-1)} .
$
If $\|b\|_{\BMO_\lambda}=0$, then $b$ is constant a.e. Write $b=c$. Then, for every interval $I\subset\Rp$,
\[
\begin{aligned}
        &\left(\frac{1}{\nu_\lambda(I)}
        \int_I e^{p\operatorname{Re}z\,b}U\,d\nu_\lambda\right)
        \left(\frac{1}{\nu_\lambda(I)}
        \int_I e^{-p'\operatorname{Re}z\,b}\sigma\,d\nu_\lambda\right)^{p-1}  \\
        &\quad =
        e^{p\operatorname{Re}z\,c}
        \left(e^{-p'\operatorname{Re}z\,c}\right)^{p-1}
        \left(\frac{1}{\nu_\lambda(I)}
        \int_I U\,d\nu_\lambda\right)
        \left(\frac{1}{\nu_\lambda(I)}
        \int_I \sigma\,d\nu_\lambda\right)^{p-1}.
\end{aligned}
\]
Since $p'(p-1)=p$, the scalar factor satisfies
\[
        e^{p\operatorname{Re}z\,c}
        \left(e^{-p'\operatorname{Re}z\,c}\right)^{p-1}
        =
        e^{p\operatorname{Re}z\,c}
        e^{-p\operatorname{Re}z\,c}
        =1 .
\]
Hence
$
        [U_z]_{A_p(d\nu_\lambda)}
        =
        [U]_{A_p(d\nu_\lambda)}
$
for all $z\in\mathbb C$. This proves the case $\|b\|_{\BMO_\lambda}=0$.

Assume now that $\|b\|_{\BMO_\lambda}>0$. Fix an interval $I$. Write
$
        b=b_I+(b-b_I).
$
Then
$$
        e^{p\operatorname{Re}z\,b}
        =
        e^{p\operatorname{Re}z\,b_I}
        e^{p\operatorname{Re}z\,(b-b_I)}
\quad{\rm
and
}\quad
        e^{-p'\operatorname{Re}z\,b}
        =
        e^{-p'\operatorname{Re}z\,b_I}
        e^{-p'\operatorname{Re}z\,(b-b_I)} .
$$
Therefore
\begin{align*}
        &\left(\frac{1}{\nu_\lambda(I)}
        \int_I e^{p\operatorname{Re}z\,b}U\,d\nu_\lambda\right)
        \left(\frac{1}{\nu_\lambda(I)}
        \int_I e^{-p'\operatorname{Re}z\,b}\sigma\,d\nu_\lambda\right)^{p-1}  \\
        &\quad =
        e^{p\operatorname{Re}z\,b_I}
        \left(e^{-p'\operatorname{Re}z\,b_I}\right)^{p-1}  
        \left(\frac{1}{\nu_\lambda(I)}
        \int_I e^{p\operatorname{Re}z\,(b-b_I)}U\,d\nu_\lambda\right)
        \left(\frac{1}{\nu_\lambda(I)}
        \int_I e^{-p'\operatorname{Re}z\,(b-b_I)}\sigma\,d\nu_\lambda\right)^{p-1}.
\end{align*}
Since $p'(p-1)=p$, we have
\[
        e^{p\operatorname{Re}z\,b_I}
        \left(e^{-p'\operatorname{Re}z\,b_I}\right)^{p-1}
        =
        e^{p\operatorname{Re}z\,b_I}
        e^{-p\operatorname{Re}z\,b_I}
        =1 .
\]
Hence
\[
\begin{aligned}
        &\left(\frac{1}{\nu_\lambda(I)}
        \int_I e^{p\operatorname{Re}z\,b}U\,d\nu_\lambda\right)
        \left(\frac{1}{\nu_\lambda(I)}
        \int_I e^{-p'\operatorname{Re}z\,b}\sigma\,d\nu_\lambda\right)^{p-1}  \\
        &\quad =
        \left(\frac{1}{\nu_\lambda(I)}
        \int_I e^{p\operatorname{Re}z\,(b-b_I)}U\,d\nu_\lambda\right)
        \left(\frac{1}{\nu_\lambda(I)}
        \int_I e^{-p'\operatorname{Re}z\,(b-b_I)}\sigma\,d\nu_\lambda\right)^{p-1}.
\end{aligned}
\]

To handle each term, we use the reverse H\"older property for $A_p$ weights and the John--Nirenberg inequality for $\BMO_\lambda$.
Quantitatively, if $U\in A_p(d\nu_\lambda)$ and $\sigma=U^{-1/(p-1)}$, then, after decreasing the constant if necessary, we may choose
$$
        \eta=c_{p,\lambda}[U]_{A_p(d\nu_\lambda)}^{-\gamma_p}, \qquad 0<\eta\le1,
$$
so that both $U$ and $\sigma$ satisfy reverse H\"older inequalities with exponent $1+\eta$.

Moreover, if $|z|$ is small enough such that
$$
\max\{p,p'\}|z|\cdot \frac{1+\eta}{\eta}\le \frac{c_\lambda}{\|b\|_{\BMO_\lambda}},
$$
then the John--Nirenberg inequality gives
$$
        \frac{1}{\nu_\lambda(I)}\int_I e^{p\operatorname{Re}z(b-b_I)\cdot \frac{1+\eta}{\eta}}\,d\nu_\lambda
        \le
        \frac{1}{\nu_\lambda(I)}\int_I \exp\left(c_\lambda \frac{|b-b_I|}{\|b\|_{\BMO_\lambda}}\right) d\nu_\lambda 
        \le C_\lambda.
$$
Since $0<\eta\le1$, we may therefore take
$$
        |z|\le c_{p,\lambda}\|b\|_{\BMO_\lambda}^{-1}[U]_{A_p(d\nu_\lambda)}^{-\gamma_p}.
$$

By H\"older's inequality, we obtain that
\begin{align*}
\frac{1}{\nu_\lambda(I)}\int_I e^{p\operatorname{Re}z(b-b_I)}U\,d\nu_\lambda
&\le 
\left(\frac{1}{\nu_\lambda(I)}\int_I e^{p\operatorname{Re}z(b-b_I)\cdot \frac{1+\eta}{\eta}}\,d\nu_\lambda \right)^\frac{\eta}{1+\eta}
\left(\frac{1}{\nu_\lambda(I)}\int_I U^{1+\eta}\,d\nu_\lambda \right)^\frac{1}{1+\eta}\\
&\le C_{\lambda} \left(\frac{1}{\nu_\lambda(I)}\int_I U^{1+\eta}\,d\nu_\lambda \right)^\frac{1}{1+\eta}\\
&\le \frac{C_{\lambda}}{\nu_\lambda(I)}\int_I U\,d\nu_\lambda.
\end{align*}

The same argument for $\sigma$, with $p'$ in place of $p$, is allowed by this smallness condition and gives
$$
\frac{1}{\nu_\lambda(I)}\int_I e^{-p'\operatorname{Re}z(b-b_I)}\sigma\,d\nu_\lambda
\leq \frac{C_\lambda}{\nu_\lambda(I)}\int_I \sigma\,d\nu_\lambda.
$$

Combining the above estimates we have
$$
        \left(\frac{1}{\nu_\lambda(I)}\int_I e^{p\operatorname{Re}zb}U\,d\nu_\lambda\right)
        \left(\frac{1}{\nu_\lambda(I)}\int_I e^{-p'\operatorname{Re}zb}\sigma\,d\nu_\lambda\right)^{p-1}\le C_{p,\lambda} \left(\frac{1}{\nu_\lambda(I)}\int_I U\,d\nu_\lambda\right)
        \left(\frac{1}{\nu_\lambda(I)}\int_I \sigma\,d\nu_\lambda\right)^{p-1}. 
$$
Taking the supremum over $I$ gives \eqref{eq:perturbed-Ap}.
\end{proof}

\section{Proof of the sharp Andersen--Kerman estimate}\label{sec:proof-riesz}

\begin{proof}[Proof of Theorem~\ref{thm:intro-riesz}]
Let $w\in A_{p,\lambda}$ and set
$
        U(x)=x^{p-2\lambda-1}w(x).
$
By Proposition~\ref{prop:exact-weight},
$
        [U]_{A_p(d\nu_\lambda)}=[w]_{A_{p,\lambda}}.
$
Write $f=xF$. By Proposition~\ref{prop:norm-identities}, it is enough to prove
$$
        \|\mathcal R_\lambda F\|_{L^p(\Rp,U\,d\nu_\lambda)} \le C_{p,\lambda}[U]_{A_p(d\nu_\lambda)}^{\gamma_p} \|F\|_{L^p(\Rp,U\,d\nu_\lambda)}.
$$
First assume that $F$ is bounded and compactly supported. For $U\in A_p(d\nu_\lambda)$, both $U\,d\nu_\lambda$
and $U^{-1/(p-1)}d\nu_\lambda$ are locally finite and $\sigma$-finite. Since bounded compactly supported functions are
dense in the corresponding weighted $L^p$ spaces, the uniform estimate below then extends the operator by
completion. Let $\sigma=U^{-1/(p-1)}$. By duality, it suffices to estimate
$$
        \left|\int_0^\infty \mathcal R_\lambda F(x)G(x)\,d\nu_\lambda(x)\right|
$$
for bounded compactly supported $G$ with $\|G\|_{L^{p'}(\Rp,\sigma\,d\nu_\lambda)}\le1$. By sparse domination (Theorem \ref{thm:sparse-domination}),
$$
        \left|\int_0^\infty \mathcal R_\lambda F(x)G(x)\,d\nu_\lambda(x)\right| \le C_\lambda
        \sum_{I\in\mathcal S} \langle |F|\rangle_I\langle |G|\rangle_I\nu_\lambda(I).
$$
The sharp sparse bound as in Theorem \ref{thm:weighted-sparse} further gives
$$
         \sum_{I\in\mathcal S} \langle |F|\rangle_I\langle |G|\rangle_I\nu_\lambda(I) \le
        C_{p,\lambda}[U]_{A_p(d\nu_\lambda)}^{\gamma_p} \|F\|_{L^p(\Rp,U\,d\nu_\lambda)}.
$$
Taking the supremum over all such $G$ proves the estimate for $\mathcal R_\lambda$. The norm identities then give
$$
        \|R_\lambda f\|_{L^p(\Rp,w\,dx)} \le C_{p,\lambda}[w]_{A_{p,\lambda}}^{\gamma_p} \|f\|_{L^p(\Rp,w\,dx)}.
$$
The proof is complete.
\end{proof}

\section{Sharp weighted estimate for the Bessel Riesz commutator}\label{sec:proof-comm}

The commutator estimate follows from the same conjugation and the Coifman--Rochberg--Weiss Cauchy
method \cite{CRW76}, in the sharp weighted form used in \cite{CPP12}.

\begin{theorem}\label{thm:conj-comm}
Let $1<p<\infty$, $U\in A_p(d\nu_\lambda)$, and let $b\in\BMO_\lambda$ be real-valued. Then, for every $F\in L^p(\Rp,U\,d\nu_\lambda)$,
\begin{align}\label{eq:conj-comm}
        \|[b,\mathcal R_\lambda]F\|_{L^p(\Rp,U\,d\nu_\lambda)} \le C_{p,\lambda}\|b\|_{\BMO_\lambda}
        [U]_{A_p(d\nu_\lambda)}^{2\gamma_p} \|F\|_{L^p(\Rp,U\,d\nu_\lambda)}.
\end{align}
The same estimate holds for complex-valued symbols after applying the real-valued estimate to their real and imaginary parts.
\end{theorem}

\begin{proof}
For a real-valued $b\in\BMO_\lambda$, put
$$
        b_N=\max\{-N,\min\{b,N\}\}.
$$
The truncation map is Lipschitz, and therefore
$$
        \|b_N\|_{\BMO_\lambda}\le C\|b\|_{\BMO_\lambda}.
$$
By John--Nirenberg, $b\in L^q_{\mathrm{loc}}(\Rp,\nu_\lambda)$ for every finite $q$, in particular
$b_N\to b$ in $L^2_{\mathrm{loc}}(\Rp,\nu_\lambda)$. If $F$ and $G$ are bounded and compactly
supported, the $L^2(\Rp, d\nu_\lambda)$ boundedness of $\mathcal R_\lambda$ gives
\begin{align}
        \langle [b_N,\mathcal R_\lambda]F,G\rangle_{\nu_\lambda} &=\langle \mathcal R_\lambda F,b_N
        G\rangle_{\nu_\lambda} -\langle \mathcal R_\lambda(b_N F),G\rangle_{\nu_\lambda}, \notag
\end{align}
and both terms converge to their analogues with $b$ in place of $b_N$. Hence $[b_N,\mathcal R_\lambda]F$ converges to $[b,\mathcal
R_\lambda]F$ in the sense of distributions tested against compactly supported bounded functions. A uniform weighted
$L^p(\Rp,U\,d\nu_\lambda)$ bound for the bounded symbols $b_N$ then passes to $b$: by reflexivity, a bounded subsequence has a weak limit in
$L^p(\Rp,U\,d\nu_\lambda)$, the distributional limit identifies it with $[b,\mathcal R_\lambda]F$, and the norm is lower semicontinuous. Thus
it is enough to prove the estimate for $b\in L^\infty\cap\BMO_\lambda$ and bounded compactly supported $F$. The passage to general $F\in
L^p(\Rp,U\,d\nu_\lambda)$ is by density, as in Theorem~\ref{thm:intro-riesz}.

For complex $z$ define
$$
        \Phi(z)F=e^{zb}\mathcal R_\lambda(e^{-zb}F).
$$
Then $\Phi$ is analytic near $z=0$ and
$$
        \Phi'(0)F=[b,\mathcal R_\lambda]F.
$$
By Cauchy's formula, for every $r>0$ small enough,
$$
        [b,\mathcal R_\lambda]F =\frac{1}{2\pi i}\int_{|z|=r}\frac{\Phi(z)F}{z^2}\,dz.
$$
Hence
\begin{align}\label{eq:cauchy-norm}
        \|[b,\mathcal R_\lambda]F\|_{L^p(\Rp,U\,d\nu_\lambda)} \le \frac{1}{r}
        \sup_{|z|=r}\|e^{zb}\mathcal R_\lambda(e^{-zb}F)\|_{L^p(\Rp,U\,d\nu_\lambda)}.
\end{align}
For fixed $z$, set
$$
        U_z=e^{p\operatorname{Re}z\,b}U.
$$
Then
\begin{align}
        \|e^{zb}\mathcal R_\lambda(e^{-zb}F)\|_{L^p(\Rp,U\,d\nu_\lambda)}^p &= \int |\mathcal R_\lambda(e^{-zb}F)|^p e^{p\operatorname{Re}z
        b}U\,d\nu_\lambda = \|\mathcal R_\lambda(e^{-zb}F)\|_{L^p(\Rp,U_z\,d\nu_\lambda)}^p. \notag
\end{align}
If $\|b\|_{\BMO_\lambda}=0$, then $b$ is constant and the commutator is zero. Assume henceforth that $\|b\|_{\BMO_\lambda}>0$ and choose
$$
        r=\frac{c_{p,\lambda}}{\|b\|_{\BMO_\lambda}[U]_{A_p(d\nu_\lambda)}^{\gamma_p}},
$$
with the constant from Lemma~\ref{lem:exp-perturb}. Then $U_z\in A_p(d\nu_\lambda)$ and $[U_z]_{A_p(d\nu_\lambda)}\le
C[U]_{A_p(d\nu_\lambda)}$ for every $|z|=r$. Applying the sharp estimate for $\mathcal R_\lambda$ with the weight $U_z$,
$$
        \|\mathcal R_\lambda(e^{-zb}F)\|_{L^p(\Rp,U_z\,d\nu_\lambda)} \le C_{p,\lambda}[U_z]_{A_p(d\nu_\lambda)}^{\gamma_p}
        \|e^{-zb}F\|_{L^p(\Rp,U_z\,d\nu_\lambda)} \le C_{p,\lambda}[U]_{A_p(d\nu_\lambda)}^{\gamma_p} \|F\|_{L^p(\Rp,U\,d\nu_\lambda)}.
$$
Here
$$
        |e^{-zb}F|^p U_z=|F|^p e^{-p\operatorname{Re}z b}e^{p\operatorname{Re}z b}U =|F|^p U.
$$
Putting this into \eqref{eq:cauchy-norm} gives
$$
        \|[b,\mathcal R_\lambda]F\|_{L^p(\Rp,U\,d\nu_\lambda)} \le \frac{C_{p,\lambda}}{r} [U]_{A_p(d\nu_\lambda)}^{\gamma_p}
        \|F\|_{L^p(\Rp,U\,d\nu_\lambda)} \le C_{p,\lambda}\|b\|_{\BMO_\lambda} [U]_{A_p(d\nu_\lambda)}^{2\gamma_p} \|F\|_{L^p(\Rp,U\,d\nu_\lambda)}.
$$
For a complex-valued symbol, write $b=b_1+ib_2$ with $b_1,b_2$ real-valued. Then
$$
        [b,\mathcal R_\lambda]=[b_1,\mathcal R_\lambda] +i[b_2,\mathcal R_\lambda],
$$
and $\|b_j\|_{\BMO_\lambda}\le \|b\|_{\BMO_\lambda}$ for $j=1,2$. Applying the real-valued estimate
to $b_1$ and $b_2$ gives the same bound, up to changing the constant.
\end{proof}

\begin{proof}[Proof of Theorem~\ref{thm:intro-comm}]
Set $U(x)=x^{p-2\lambda-1}w(x)$ and write $f=xF$. By Proposition~\ref{prop:exact-weight},
$$
        [U]_{A_p(d\nu_\lambda)}=[w]_{A_{p,\lambda}}.
$$
By the commutator norm identity \eqref{eq:norm-comm} and Theorem~\ref{thm:conj-comm},
\begin{align}
        \|[b,R_\lambda]f\|_{L^p(\Rp,w\,dx)} &=\|[b,\mathcal R_\lambda]F\|_{L^p(\Rp,U\,d\nu_\lambda)} \le
        C_{p,\lambda}\|b\|_{\BMO_\lambda} [U]_{A_p(d\nu_\lambda)}^{2\gamma_p} \|F\|_{L^p(\Rp,U\,d\nu_\lambda)} \notag \\
        &= C_{p,\lambda}\|b\|_{\BMO_\lambda} [w]_{A_{p,\lambda}}^{2\gamma_p} \|f\|_{L^p(\Rp,w\,dx)}. \notag
\end{align}
The proof is complete.
\end{proof}

We also address the bounds for iterated commutators.
\begin{corollary}\label{cor:iterated}
Let $\lambda>-1/2$, $\lambda\ne0$, $1<p<\infty$, $w\in A_{p,\lambda}$, let $m\ge1$ be an
integer, and let $b\in\BMO_\lambda$ be real-valued. Define
$$
        R_{\lambda,b}^0=R_\lambda, \qquad R_{\lambda,b}^{k}=[b,R_{\lambda,b}^{k-1}],\qquad k\ge1.
$$
Then, with $\gamma_p=\max\{1,1/(p-1)\}$,
\begin{align}\label{eq:iterated-comm-bound}
        \|R_{\lambda,b}^{m}f\|_{L^p(\Rp,w\,dx)} \le C_{m,p,\lambda}\|b\|_{\BMO_\lambda}^{m} [w]_{A_{p,\lambda}}^{(m+1)\gamma_p}
        \|f\|_{L^p(\Rp,w\,dx)}.
\end{align}
\end{corollary}

\begin{proof}
Set $U=x^{p-2\lambda-1}w$ and write $f=xF$. By Propositions~\ref{prop:exact-weight} and~\ref{prop:norm-identities}, it suffices to prove
the corresponding estimate for the iterated commutators of $\mathcal R_\lambda$ on $L^p(\Rp,U\,d\nu_\lambda)$.

Assume first that $b$ is bounded and $\|b\|_{\BMO_\lambda}>0$. Put
$$
        \Phi(z)F=e^{zb}\mathcal R_\lambda(e^{-zb}F).
$$
The $m$-th derivative of $\Phi(z)F$ at the origin is exactly the $m$-th iterated commutator of $\mathcal R_\lambda$ with the symbol $b$.

Define
$$
       \mathcal  R_{\lambda,b}^0=\mathcal R_\lambda, \qquad \mathcal  R_{\lambda,b}^{k}=[b,\mathcal  R_{\lambda,b}^{k-1}],\qquad k\ge1.
$$
Cauchy's formula gives
$$
        \|\mathcal R_{\lambda,b}^{m}(F)\|_{L^p(\Rp,U\,d\nu_\lambda)} \le C_m r^{-m}
        \sup_{|z|=r} \|e^{zb}\mathcal R_\lambda(e^{-zb}F)\|_{L^p(\Rp,U\,d\nu_\lambda)}.
$$
Choose
$$
        r=\frac{c_{p,\lambda}} {\|b\|_{\BMO_\lambda}[U]_{A_p(d\nu_\lambda)}^{\gamma_p}},
$$
with $c_{p,\lambda}$ as in Lemma~\ref{lem:exp-perturb}. For $|z|=r$ the perturbed weights
$U_z=e^{p\operatorname{Re}z\,b}U$ satisfy $[U_z]_{A_p(d\nu_\lambda)}\lesssim [U]_{A_p(d\nu_\lambda)}$.
Applying the sharp estimate for $\mathcal R_\lambda$ with weight $U_z$ gives
\begin{align}
        \|e^{zb}\mathcal R_\lambda(e^{-zb}F)\|_{L^p(\Rp,U\,d\nu_\lambda)} &=\|\mathcal R_\lambda(e^{-zb}F)\|_{L^p(\Rp,U_z\,d\nu_\lambda)} \notag 
        \le C_{p,\lambda}[U]_{A_p(d\nu_\lambda)}^{\gamma_p} \|F\|_{L^p(\Rp,U\,d\nu_\lambda)}. \notag
\end{align}
Since $r^{-m}\approx \|b\|_{\BMO_\lambda}^{m}[U]_{A_p(d\nu_\lambda)}^{m\gamma_p}$, the conjugated estimate follows. If
$\|b\|_{\BMO_\lambda}=0$, all positive-order commutators vanish. General real-valued symbols follow by the truncation and weak compactness
argument used in Theorem~\ref{thm:conj-comm}. The norm identities return the estimate to the original variables.
\end{proof}

\section{Comparison with usual Bessel weights and the \texorpdfstring{$\widetilde A$}{A-tilde} classes}\label{sec:not-usual}

Three related weight conditions occur in the Bessel setting.

First, the usual weighted theory on the Bessel space of homogeneous type uses
$$
        (\Rp,|x-y|,dm_\lambda), \qquad dm_\lambda(x)=x^{2\lambda}\,dx,
$$
and estimates operators in
$$
        L^p(w\,dm_\lambda) =L^p(\Rp,w(x)x^{2\lambda}\,dx).
$$
The Andersen--Kerman theorem estimates $R_\lambda$ in
$$
        L^p(\Rp,w(x)\,dx).
$$
These are different norms. To rewrite the usual Bessel theorem in the Andersen--Kerman norm, write
$$
        w(x)\,dx=(w(x)x^{-2\lambda})\,dm_\lambda(x).
$$
Thus the usual theorem gives the condition
$$
        w(x)x^{-2\lambda}\in A_p(dm_\lambda).
$$

Second, the papers \cite{LLLS,LLSW} use another Muckenhoupt-type class. In the notation of \cite{LLSW}, for $1<p<\infty$,
\begin{align}
        [w]_{\widetilde A_{p,\beta}} :=\sup_{I\subset\Rp} \left(\frac{1}{\nu_\beta(I)}\int_I w(x)\,dx\right)
        \left(\frac{1}{\nu_\beta(I)}\int_I x^{(2\beta+1)p'}w(x)^{-1/(p-1)}\,dx\right)^{p-1}, \notag
\end{align}
where $d\nu_\beta(x)=x^{2\beta+1}\,dx$. For $-1/2<\lambda<0$, the parameter $\beta=\lambda-\frac12$ lies outside the usual
Bessel-operator range but the displayed weight characteristic still makes sense as a power-measure quantity, since
$2\beta+1=2\lambda>-1$. Taking $\beta=\lambda-\frac12$ gives $d\nu_{\lambda-1/2}=dm_\lambda$. Hence
$$
        [w]_{\widetilde A_{p,\lambda-\frac12}} =[w(x)x^{-2\lambda}]_{A_p(dm_\lambda)}.
$$
Thus the shifted class $\widetilde A_{p,\lambda-\frac12}$ is precisely the class obtained by converting the
usual Bessel $A_p(dm_\lambda)$ condition to the Lebesgue norm $L^p(\Rp,w\,dx)$.

The Andersen--Kerman class is larger than this shifted class.

\begin{proposition}\label{prop:tilde-contained-AK}
Let $\lambda>-1/2$ and $1<p<\infty$. Then
$$
        \widetilde A_{p,\lambda-\frac12}\subset A_{p,\lambda}, \qquad [w]_{A_{p,\lambda}}
        \lesssim_{p,\lambda} [w]_{\widetilde A_{p,\lambda-\frac12}}.
$$
The inclusion is strict.
\end{proposition}

\begin{proof}
Put $d\mu_\lambda(x)=x^{2\lambda}\,dx$ and $d\nu_\lambda(x)=x^{2\lambda+1}\,dx$. Let $I=(a,b)\subset \Rp$. Since $2\lambda>-1$, the
$\mu_\lambda$-average of $x$ over $I$ is comparable with the right endpoint $b$:
$$
        \frac{\nu_\lambda(I)}{\mu_\lambda(I)} =\frac{\int_I x\,d\mu_\lambda(x)}{\mu_\lambda(I)} \approx_\lambda b.
$$
If $a>b/2$ this follows at once, while if $a\le b/2$ one integrates over $(b/2,b)$ for the lower bound. Hence
$x\le C_\lambda \nu_\lambda(I)/\mu_\lambda(I)$ for all $x\in I$. Therefore
$$
        \int_I x^p w(x)\,dx \le C_\lambda^p \left(\frac{\nu_\lambda(I)}{\mu_\lambda(I)}\right)^p \int_I w(x)\,dx.
$$
Using this estimate in the definition of $A_{p,\lambda}$ gives
\begin{align}
        &\left(\frac{1}{\nu_\lambda(I)}\int_I x^p w(x)\,dx\right)
        \left(\frac{1}{\nu_\lambda(I)}\int_I x^{2\lambda p'}w(x)^{-1/(p-1)}\,dx\right)^{p-1} \notag \\
        &\qquad \le C_{p,\lambda} \left(\frac{1}{\mu_\lambda(I)}\int_I w(x)\,dx\right)
        \left(\frac{1}{\mu_\lambda(I)}\int_I x^{2\lambda p'}w(x)^{-1/(p-1)}\,dx\right)^{p-1}. \notag
\end{align}
Taking the supremum over $I$ proves the characteristic estimate and the inclusion.

The inclusion is strict by power weights. If $w(x)=x^\alpha$, then a direct calculation gives
$$
        w\in A_{p,\lambda} \quad\Longleftrightarrow\quad -1-p<\alpha<2\lambda p+p-1,
$$
and
$$
        w\in \widetilde A_{p,\lambda-\frac12} \quad\Longleftrightarrow\quad -1<\alpha<2\lambda p+p-1.
$$
Thus $w(x)=x^{-1-\varepsilon}$, $0<\varepsilon<p$, belongs to $A_{p,\lambda}$ but not to $\widetilde A_{p,\lambda-\frac12}$.
\end{proof}

Thus the quantitative Riesz-transform estimate in \cite{LLSW} is formulated for weights in the shifted class $\widetilde
A_{p,\lambda-\frac12}$, whereas our present estimate holds on the full Andersen--Kerman class $A_{p,\lambda}$,
including singular weights at the origin outside $\widetilde A_{p,\lambda-\frac12}$.

There is also a useful comparison with the unshifted class $\widetilde A_{p,\lambda}$, the class
associated with the Bessel maximal operator in \cite{LLLS}. For power weights,
$$
        x^\alpha\in \widetilde A_{p,\lambda} \quad\Longleftrightarrow\quad -1<\alpha<2\lambda p+2p-1.
$$
Consequently $A_{p,\lambda}$ and $\widetilde A_{p,\lambda}$ do not contain one another. The weight $x^{-1-\varepsilon}$,
$0<\varepsilon<p$, belongs to $A_{p,\lambda}$ but not to $\widetilde A_{p,\lambda}$. On the other hand, $x^{2\lambda
p+p-1+\varepsilon}$, $0<\varepsilon<p$, belongs to $\widetilde A_{p,\lambda}$ but not to $A_{p,\lambda}$. Thus $A_{p,\lambda}$ and $\widetilde A_{p,\lambda}$ are distinct weight classes. The former is adapted to the Lebesgue weighted Bessel Riesz
transform, while the latter is tied to the maximal-operator side of the Bessel setting \cite{LLLS}.

\section{Sharpness of the powers}\label{sec:sharpness}

The powers in Theorems~\ref{thm:intro-riesz} and~\ref{thm:intro-comm} are optimal.
The test functions, the variation of the symbol, and the nonconstant part of the conjugated weights in the lower-bound examples can all be localized in a fixed compact subinterval of \(\Rp\).
On such an interval, the measure \(d\nu_\lambda=x^{2\lambda+1}\,dx\) is comparable to Lebesgue measure.
Moreover, \eqref{eq:AK-local-sign-conj} shows that the conjugated Bessel Riesz kernel has a fixed sign and dominates, on one side of the diagonal, the corresponding Hilbert-transform kernel.
Thus the classical sharpness examples for the Hilbert transform and its commutator transfer to the present Bessel setting.

The next lemma permits passage between Lebesgue measure and $\nu_\lambda$ without changing powers of the weight characteristic.

\begin{lemma}\label{lem:compact-localization}
Let $K\Subset\Rp$ and choose an interval $K^*$ with $K\Subset K^*\Subset\Rp$. Then the following
statements hold with constants depending only on $p,\lambda,K,K^*$.

\noindent\emph{(i)} If $F$ is supported in $K$, then, for every weight $W$,
$
        \|F\|_{L^p(\Rp,W\,d\nu_\lambda)}\approx \|F\|_{L^p(\Rp,W\,dx)}.
$

\noindent\emph{(ii)} If $W=1$ on $\Rp\setminus K$, then
$
        [W]_{A_p(d\nu_\lambda)}\approx [W]_{A_p(dx)}.
$

\noindent\emph{(iii)} If $b$ is constant on $\Rp\setminus K$, then
$
        \|b\|_{\BMO_\lambda}\approx \|b\|_{\BMO(dx)}.
$
\end{lemma}

\begin{proof}
Part (i) follows from the upper and lower bounds for $x^{2\lambda+1}$ on $K$.

It is enough to prove one of the two implications in (ii), since the proof of the other implication is the same. Let
$
        \sigma=W^{-1/(p-1)} .
$
We first show that the \(A_p\) quantity computed with respect to \(d\nu_\lambda\) is controlled by the usual \(A_p\) quantity computed with respect to Lebesgue measure. More precisely, for each interval \(L\), we shall compare
$$
        \left(\frac1{\nu_\lambda(L)}\int_L W\,d\nu_\lambda\right)
        \left(\frac1{\nu_\lambda(L)}\int_L \sigma\,d\nu_\lambda\right)^{p-1}
$$
with the corresponding expression in which \(d\nu_\lambda\) is replaced by \(dx\). For a positive measure \(\mu\), write
$$
        a_p^{d\mu}(L,W)
        :=
        \left(\frac1{\mu(L)}\int_L W\,d\mu\right)
        \left(\frac1{\mu(L)}\int_L \sigma\,d\mu\right)^{p-1}.
$$

We now consider the following three cases.

(1). $L\cap K=\emptyset$. Then $a_p^{d\nu_\lambda}(L,W)=1$. 

(2). $L\subset K^*$. Then the  integration averages over
$L$ in terms of $dx$ and $d\nu_\lambda$ are comparable, because $x^{2\lambda+1}$ is bounded above and below on $K^*$. Hence, 
$$
        a_p^{d\nu_\lambda}(L,W) \lesssim a_p^{dx}(L,W) \le [W]_{A_p(dx)}.
$$

(3). $L\cap K\not=\emptyset\ $ and $\ L\not\subset K^*$. Since $K\Subset K^*$,
such intervals have $\nu_\lambda(L)\ge c_{K,K^*,\lambda}>0$. Let
$$
        A=\frac1{\nu_\lambda(K^*)}\int_{K^*}W\,d\nu_\lambda, \qquad B=\frac1{\nu_\lambda(K^*)}\int_{K^*}\sigma\,d\nu_\lambda.
$$
Because $W=\sigma=1$ on $K^*\setminus K$, and this set has positive $\nu_\lambda$-measure,
$A$ and $B$ are bounded below by a fixed positive constant. Also
$$
        \frac1{\nu_\lambda(L)}\int_L W\,d\nu_\lambda 
        \le \frac1{\nu_\lambda(L)} \bigg(\int_{L\backslash K} W\,d\nu_\lambda+\int_K W\,d\nu_\lambda\bigg)
        \le 1+C_{K,K^*,\lambda} \,A,
 $$
 $$\frac1{\nu_\lambda(L)}\int_L \sigma\,d\nu_\lambda 
 \le \frac1{\nu_\lambda(L)} \bigg(\int_{L\backslash K} \sigma\,d\nu_\lambda+\int_K \sigma\,d\nu_\lambda\bigg)
 \le 1+C_{K,K^*,\lambda}\,B.
$$
Consequently
$$
        a_p^{d\nu_\lambda}(L,W) \lesssim A B^{p-1}.
$$
On $K^*$, the two measures $dx$ and $d\nu_\lambda$ are comparable. Hence
$$
        AB^{p-1}\lesssim [W]_{A_p(dx)}.
$$
Therefore $[W]_{A_p(d\nu_\lambda)}\lesssim [W]_{A_p(dx)}$. Interchanging the two measures gives the reverse inequality.

For (iii), it is enough to prove one direction, namely $\|b\|_{\BMO(d\nu_\lambda)}\lesssim\|b\|_{\BMO(dx)}$. Let $b_\infty$ denote the
constant value of $b$ on $\Rp\setminus K$. 
If $L\cap K=\emptyset$, 
then $$\frac1{\nu_\lambda(L)}\int_L |b-b_L|d\nu_\lambda =0.$$
If $L\subset
K^*$, then the integration averages over $L$ in terms of $dx$ and $d\nu_\lambda$  are comparable. For any constant $a$,
$$
        \frac1{\nu_\lambda(L)}\int_L |b-a|\,d\nu_\lambda \lesssim_{K^*,\lambda} \frac1{|L|}\int_L |b-a|\,dx.
$$
Taking the infimum over $a$ and using the equivalent infimum form of mean oscillation gives the desired comparison on such intervals.

It suffices to consider $L\cap K\not=\emptyset\ $ and $\ L\not\subset K^*$. Since $\nu_\lambda(L)\ge c_{K,K^*,\lambda}>0$ and $b=b_\infty$ on $\Rp\backslash K$,
$$
        \frac1{\nu_\lambda(L)}\int_L |b-b_\infty|\,d\nu_\lambda \lesssim_{K,K^*,\lambda} \int_K |b-b_\infty|\,d\nu_\lambda.
$$
The last integral is bounded by $\|b\|_{\BMO(dx)}$. Indeed, since $K^*\setminus K$ has positive Lebesgue measure and $b=b_\infty$ on $K^*\setminus K$,
$$
        |b_\infty-b_{K^*}^{dx}| 
        = \frac1{|K^*\backslash K|}\, \bigg|\int_{K^*\backslash K} b_\infty-b_{K^*}^{dx}\,dx\bigg|
        \le C_{K,K^*}\frac1{|K^*|}\int_{K^*}|b-b_{K^*}^{dx}|\,dx \le C_{K,K^*}\|b\|_{\BMO(dx)}.
$$
Thus
$$
        \int_K |b-b_\infty|\,d\nu_\lambda \lesssim_{K,K^*,\lambda} \int_{K^*}|b-b_{K^*}^{dx}|\,dx+|K^*|\|b\|_{\BMO(dx)} \lesssim_{K,K^*,\lambda} \|b\|_{\BMO(dx)}.
$$
Using $b_\infty$ as the comparison constant in the mean oscillation over $L$ gives the required bound. The reverse BMO
inequality is the same argument with $dx$ and $d\nu_\lambda$ interchanged.
\end{proof}

The estimates in this subsection are local forms of the standard sharp examples for the Hilbert transform, due to Hunt--Muckenhoupt--Wheeden and Petermichl
\cite{HuntMuckenhouptWheeden1973,Pet07}. To implement this, we first need the following auxiliary lemma. It places the two testing intervals on different sides of the origin and hence the points \(x\) and \(y\) which occur in the kernel are always separated by \(0\). In particular, \(x\ne y\), and the kernel is used as an ordinary function. No principal-value cancellation is needed.

\begin{lemma}\label{lem:one-sided-model}
Let $1<p<\infty$ and
$$
        \gamma_p=\max\left\{1,\frac1{p-1}\right\}.
$$
For $0<\varepsilon<1/4$ there are weights $W_\varepsilon$ on $(-1,1)$, non-negative functions $F_\varepsilon$,
and the symbol $b(t)=\log |t|$ such that $[W_\varepsilon]_{A_p(-1,1)}\to\infty$ and the following lower bounds
hold. Here $[W]_{A_p(-1,1)}$ denotes the Lebesgue $A_p$ characteristic after setting $W(t)=1$ for $|t|\ge1$. For the power weights below this is equivalent, with constants
independent of $\varepsilon$, to taking the supremum over subintervals of $(-1,1)$.

If $1<p\le2$, set
$$
        W_\varepsilon(t)=|t|^{(p-1)(1-\varepsilon)}, \qquad F_\varepsilon(t)=|t|^{-1+\varepsilon}\one_{(-1,0)}(t).
$$
For $0<s<1/16$ define
$$
        T_\varepsilon^-(s) =\int_0^1\frac{u^{-1+\varepsilon}}{s+u}\,du, \qquad C_\varepsilon^-(s)
        =\left| \int_0^1\frac{(\log s-\log u)u^{-1+\varepsilon}}{s+u}\,du \right|.
$$
Then, with $A_\varepsilon=[W_\varepsilon]_{A_p(-1,1)}$,
\begin{align}\label{eq:model-lower-left}
        \|T_\varepsilon^-\|_{L^p((0,1/16),W_\varepsilon(s)\,ds)} \ge c_p A_\varepsilon^{\gamma_p}
        \|F_\varepsilon\|_{L^p((-1,1),W_\varepsilon(t)\,dt)}
\end{align}
and
\begin{align}\label{eq:model-comm-lower-left}
        \|C_\varepsilon^-\|_{L^p((0,1/16),W_\varepsilon(s)\,ds)} \ge c_p A_\varepsilon^{2\gamma_p}
        \|F_\varepsilon\|_{L^p((-1,1),W_\varepsilon(t)\,dt)}.
\end{align}

If $p>2$, set
$$
        W_\varepsilon(t)=|t|^{\varepsilon-1}, \qquad F_\varepsilon(t)=\one_{(0,1)}(t).
$$
For $0<s<1/16$ define
$$
        T_\varepsilon^+(s) =\int_0^1\frac{du}{s+u}, \qquad C_\varepsilon^+(s) =\left| \int_0^1\frac{\log s-\log u}{s+u}\,du \right|.
$$
Then, with $A_\varepsilon=[W_\varepsilon]_{A_p(-1,1)}$,
\begin{align}\label{eq:model-lower-right}
        \|T_\varepsilon^+\|_{L^p((0,1/16),W_\varepsilon(s)\,ds)} \ge c_p A_\varepsilon^{\gamma_p}
        \|F_\varepsilon\|_{L^p((-1,1),W_\varepsilon(t)\,dt)}
\end{align}
and
\begin{align}\label{eq:model-comm-lower-right}
        \|C_\varepsilon^+\|_{L^p((0,1/16),W_\varepsilon(s)\,ds)} \ge c_p A_\varepsilon^{2\gamma_p}
        \|F_\varepsilon\|_{L^p((-1,1),W_\varepsilon(t)\,dt)}.
\end{align}
Moreover, the same lower bounds are stable under replacement of the kernel. More
precisely, let $L(s,u)$ be a non-negative measurable kernel such that
\begin{align}\label{eq:model-kernel-comparable}
        \frac{c_0}{s+u}\le L(s,u)\le\frac{C_0}{s+u}, \qquad 0<s<\frac{1}{16},\quad 0<u<1.
\end{align}
For the commutator estimates with this replacement kernel, assume in addition that
$0<\varepsilon<\varepsilon_0(c_0,C_0)$.
If $1<p\le2$ and
$$
        T_{\varepsilon,L}^-(s)=\int_0^1L(s,u)u^{-1+\varepsilon}\,du, \qquad
        C_{\varepsilon,L}^-(s)= \left|\int_0^1(\log s-\log u)L(s,u)u^{-1+\varepsilon}\,du\right|,
$$
then \eqref{eq:model-lower-left}--\eqref{eq:model-comm-lower-left} remain true with $T_\varepsilon^-,C_\varepsilon^-$
replaced by $T_{\varepsilon,L}^-,C_{\varepsilon,L}^-$. If $p>2$ and
$$
        T_{\varepsilon,L}^+(s)=\int_0^1L(s,u)\,du, \qquad C_{\varepsilon,L}^+(s)= \left|\int_0^1(\log s-\log u)L(s,u)\,du\right|,
$$
then \eqref{eq:model-lower-right}--\eqref{eq:model-comm-lower-right} remain true with $T_\varepsilon^+,C_\varepsilon^+$ replaced by
$T_{\varepsilon,L}^+,C_{\varepsilon,L}^+$. The constants may also depend on $c_0$ and $C_0$.
\end{lemma}

\begin{proof}
For a power weight $|t|^\alpha$ on $(-1,1)$, with $-1<\alpha<p-1$,
$$
        [|t|^\alpha]_{A_p(-1,1)} \approx_p \frac1{\alpha+1} \left(\frac{p-1}{p-1-\alpha}\right)^{p-1}.
$$
For $1<p\le2$, this gives $A_\varepsilon\approx_p\varepsilon^{1-p}$. Moreover,
$$
        \|F_\varepsilon\|_{L^p((-1,1),W_\varepsilon\,dt)}^p
        =\int_0^1 u^{-1+\varepsilon}\,du
        =\varepsilon^{-1}.
$$
For $0<s<1/16$,
$$
        T_\varepsilon^-(s) =s^{-1+\varepsilon} \int_0^{1/s}\frac{v^{-1+\varepsilon}}{1+v}\,dv \ge c\,\varepsilon^{-1}s^{-1+\varepsilon}.
$$
Also
\begin{align}
        \int_0^1\frac{(\log s-\log u)u^{-1+\varepsilon}}{s+u}\,du &= s^{-1+\varepsilon} \int_0^{1/s}
        \frac{\log(1/v)v^{-1+\varepsilon}}{1+v}\,dv. \notag
\end{align}
Notice that for $0<\varepsilon<1/4$,
$$
\left|\int_0^{1/s}
        \frac{\log(1/v)v^{-1+\varepsilon}}{1+v}\,dv\right|
        \ge\left|\int_0^1
        \frac{\log(1/v)v^{-1+\varepsilon}}{1+v}\,dv\right|
        -\left|\int_1^{1/s}
        \frac{\log(1/v)v^{-1+\varepsilon}}{1+v}\,dv\right|
        \ge \frac{1}{2\varepsilon^2}-\frac{16}{9}
        \ge  \frac{1}{4\varepsilon^2}.
$$
Hence, for $0<\varepsilon<1/4$,
$$
        C_\varepsilon^-(s)\ge c\,\varepsilon^{-2}s^{-1+\varepsilon}, \qquad 0<s<1/16.
$$

Since
$$
        \int_0^{1/16}s^{-1+\varepsilon}\,ds\approx\varepsilon^{-1},
$$
it follows that
$$\|T_\varepsilon^-\|^p_{L^p((0,1/16),W_\varepsilon(s)\,ds)}
=\int_0^{\frac{1}{16}} |T_\varepsilon^-(s)|^p\, s^{(p-1)(1-\varepsilon)}\, ds
\gtrsim \varepsilon^{-p}\int_0^{\frac{1}{16}}s^{-1+\varepsilon}\,ds
\gtrsim \varepsilon^{-p-1},
$$
and
$$\|C_\varepsilon^-\|^p_{L^p((0,1/16),W_\varepsilon(s)\,ds)}
=\int_0^{\frac{1}{16}} |C_\varepsilon^-(s)|^p\, s^{(p-1)(1-\varepsilon)}\, ds
\gtrsim \varepsilon^{-2p}\int_0^{\frac{1}{16}}s^{-1+\varepsilon}\,ds
\gtrsim \varepsilon^{-2p-1}.
$$

Consequently, with $A_\varepsilon\approx_p\varepsilon^{1-p}$ and 
$\gamma_p=\max\{1,\frac1{p-1}\}=\frac1{p-1}$, the corresponding ratios are
$$\|T_\varepsilon^-\|_{L^p((0,1/16),W_\varepsilon(s)\,ds)}/ 
\|F_\varepsilon\|_{L^p((-1,1),W_\varepsilon(t)\,dt)}
\gtrsim \varepsilon^{-1}\approx A_\varepsilon^{\gamma_p},
$$
and
$$\|C_\varepsilon^-\|_{L^p((0,1/16),W_\varepsilon(s)\,ds)}/ 
\|F_\varepsilon\|_{L^p((-1,1),W_\varepsilon(t)\,dt)}
\gtrsim \varepsilon^{-2}\approx A_\varepsilon^{2\gamma_p},
$$
which proves the two lower bounds \eqref{eq:model-lower-left} and \eqref{eq:model-comm-lower-left}.

For $p>2$, one has $A_\varepsilon\approx_p\varepsilon^{-1}$ and
$$
        \|F_\varepsilon\|_{L^p(W_\varepsilon\,dt)}^p =\int_0^1u^{\varepsilon-1}\,du =\varepsilon^{-1}.
$$
For $0<s<1/16$,
$$
        T_\varepsilon^+(s)=\int_0^1\frac{du}{s+u} \ge c\log\frac1s.
$$

Furthermore, 
\begin{align*}
        \left|\int_1^{1/s}\frac{\log(1/v)}{1+v}\,dv\right|
        &=
        \int_1^{1/s}\frac{\log v}{1+v}\,dv  
        \ge
        \frac12\int_1^{1/s}\frac{\log v}{v}\,dv
        =
        \frac14 \bigg(\log\frac1s\bigg)^2 ,
\end{align*}
where we used $1+v\le 2v$ for $v\ge 1$. Also
$$
        \int_0^1\frac{\log(1/v)}{1+v}\,dv
        =
        \frac{\pi^2}{12}.
$$
Therefore
$$
        C_\varepsilon^+(s)
        =
        \left|\int_0^{1/s}\frac{\log(1/v)}{1+v}\,dv\right|  
        \ge
        \left|\int_1^{1/s}\frac{\log(1/v)}{1+v}\,dv\right|
        -
        \left|\int_0^1\frac{\log(1/v)}{1+v}\,dv\right|  
        \ge
        \frac14\left(\log\frac1s\right)^2
        -
        \frac{\pi^2}{12}.
$$

For $0<s<1/16$, 
$$
        \frac14\left(\log\frac1s\right)^2
        -
        \frac{\pi^2}{12}
        \ge
        \frac18\left(\log\frac1s\right)^2 .
$$
Hence, 
$$
        C_\varepsilon^+(s)
        \ge
        \frac18\left(\log\frac1s\right)^2 .
$$

Finally,
$$
        \int_0^{1/16}\left(\log\frac1s\right)^q s^{\varepsilon-1}\,ds \approx_q \varepsilon^{-q-1}, \qquad q>0,
$$
uniformly for $0<\varepsilon<1/16$. Then it follows that
$$
\|T_\varepsilon^+\|^p_{L^p((0,1/16),W_\varepsilon(s)\,ds)}
=\int_0^{\frac1{16}} |T_\varepsilon^+(s)|^p\,s^{\varepsilon-1}\,ds
\gtrsim \varepsilon^{-p-1},
$$
and
$$
\|C_\varepsilon^+\|^p_{L^p((0,1/16),W_\varepsilon(s)\,ds)}
=\int_0^{\frac1{16}} |C_\varepsilon^+(s)|^p\,s^{\varepsilon-1}\,ds
\gtrsim \varepsilon^{-2p-1}.
$$

The ratios we want are therefore bounded below by $\varepsilon^{-1}$ and $\varepsilon^{-2}$.
Since $\gamma_p=1$ for $p>2$, this proves the two lower bounds \eqref{eq:model-lower-right} and \eqref{eq:model-comm-lower-right}.

We now explain why the same lower bounds remain true for every kernel satisfying
\eqref{eq:model-kernel-comparable}. For the estimates of \(T_\varepsilon^\pm\), no cancellation is used. The proof only uses the lower bound in \eqref{eq:model-kernel-comparable}. Hence the previous lower estimates remain valid, with constants changed by a factor depending on \(c_0\).

For the commutator estimates, we split the integral into the two regions
\[
        0<u<s
        \qquad\text{and}\qquad
        s<u<1 .
\]
Consider first the case \(1<p\le 2\). In this case the main contribution comes from \(0<u<s\). On this region the kernel has a fixed sign, and the lower bound in \eqref{eq:model-kernel-comparable} gives
\[
        \left|\int_0^s (\log s-\log u)L(s,u)u^{-1+\varepsilon}\,du\right|
        \ge
        c_0 s^{-1+\varepsilon}\int_0^1
        \frac{\log(1/v)v^{-1+\varepsilon}}{1+v}\,dv
        \ge
        \frac{1}{2}c_0\varepsilon^{-2}s^{-1+\varepsilon}.
\]
On the other hand, the contribution from \(s<u<1\) is an error term. Using the upper bound in \eqref{eq:model-kernel-comparable}, we have
\[
        \left|\int_s^1 (\log s-\log u)L(s,u)u^{-1+\varepsilon}\,du\right|
        \le
        C_0 s^{-1+\varepsilon}\int_1^{1/s}
        \frac{\log v\ v^{-1+\varepsilon}}{1+v}\,dv
        \le
        \frac{16}{9}C_0s^{-1+\varepsilon}.
\]
Thus, if \(\varepsilon>0\) is chosen small enough, the first term is larger than twice the second term. Therefore the full commutator term is bounded from below by
\[
        C_{\varepsilon,L}^-(s)
        =\left|\int_0^1(\log s-\log u)L(s,u)u^{-1+\varepsilon}\,du\right|
        \ge c\,\varepsilon^{-2}s^{-1+\varepsilon}.
\]
with a sufficiently small constant \(c\).

Now consider the case \(p>2\). In this case the main contribution comes from \(s<u<1\). The lower bound in \eqref{eq:model-kernel-comparable} gives
\[
        \left|\int_s^1 (\log s-\log u)L(s,u)\,du\right|
        \ge
        c_0 \int_1^{1/s}
        \frac{\log v}{1+v}\,dv
        \ge
        \frac{c_0}{4}\left(\log\frac1s\right)^2 .
\]
The remaining part, coming from \(0<u<s\), is bounded in absolute value by a constant:
\[
        \left|\int_0^s (\log s-\log u)L(s,u)\,du\right|
        \le
        C_0 \int_0^1
        \frac{\log(1/v)}{1+v}\,dv
        \le \frac{\pi^2}{12}C_0 .
\]
Hence, after restricting \(s\) to a sufficiently small interval, for example
\[
        0<s<s_0,
        \qquad
        s_0=\min\left\{\frac1{16},
        \exp\left(-\pi\sqrt{\frac{C_0}{c_0}}\right)\right\}.
\]
The logarithmic term is then larger than twice the error term. Therefore the full commutator term satisfies
\[
        C_{\varepsilon,L}^+(s)= \left|\int_0^1(\log s-\log u)L(s,u)\,du\right|
        \ge
        c\left(\log\frac1s\right)^2,
\]
with a sufficiently small constant \(c\).

Integrating over $(0,s_0)$ gives the required lower estimates. The constants may now also depend on the comparison constants \(c_0\) and \(C_0\) in \eqref{eq:model-kernel-comparable}.
\end{proof}

\begin{proposition}\label{prop:sharpness-powers}
Let $\lambda>-1/2$, $\lambda\ne0$, $1<p<\infty$, and
$$
        \gamma_p=\max\left\{1,\frac{1}{p-1}\right\}.
$$
If an estimate of the form
$$
        \|R_\lambda f\|_{L^p(\Rp,w\,dx)} \le C_{p,\lambda}[w]_{A_{p,\lambda}}^{\theta} \|f\|_{L^p(\Rp,w\,dx)}
$$
holds for all $w\in A_{p,\lambda}$ and all compactly supported $f\in L^p(\Rp,w\,dx)$, with $C_{p,\lambda}$ independent of $w$ and $f$, then $\theta\ge\gamma_p$. If an estimate of the form
$$
        \|[b,R_\lambda]f\|_{L^p(\Rp,w\,dx)} \le C_{p,\lambda}\|b\|_{\BMO_\lambda} [w]_{A_{p,\lambda}}^{\delta} \|f\|_{L^p(\Rp,w\,dx)}
$$
holds for all $w\in A_{p,\lambda}$, all compactly supported $f\in L^p(\Rp,w\,dx)$, and all real-valued $b\in\BMO_\lambda$, with $C_{p,\lambda}$ independent of $w$, $f$, and $b$, then $\delta\ge2\gamma_p$.
\end{proposition}

\begin{proof}
By the exact conjugation, the asserted estimates on the original weighted space $L^p(\Rp,w\,dx)$ are equivalent to the corresponding estimates for
$\mathcal R_\lambda$ on $L^p(\Rp,U\,d\nu_\lambda)$, with $[U]_{A_p(d\nu_\lambda)}=[w]_{A_{p,\lambda}}$.

Choose a compact interval $J\Subset\Rp$. By \eqref{eq:AK-local-sign-conj}, there are $s_J>1$,
$c_J>0$, and $\varepsilon_\lambda\in\{-1,1\}$ such that the one-sided lower bound holds whenever
$x,y\in J$ and $s_J^{-1}<y/x<s_J$. Fix $a\in J$ and choose $\eta>0$ so small that
$$
        K=(a-\eta,a+\eta)\Subset J
$$
and such that $s_J^{-1}<y/x<s_J$ for every $x,y\in K$.

Take one of the two families of functions constructed in Lemma~\ref{lem:one-sided-model}, using the first family when $1<p\le2$ and the second family when
$p>2$. Thus the endpoint $p=2$ is covered by the first family. Transfer it to $K$ by
$$
        U_\varepsilon(a+\eta t)=W_\varepsilon(t), \qquad F_\varepsilon^{K}(a+\eta t)=F_\varepsilon(t), \qquad -1<t<1,
$$
put $F_\varepsilon^K=0$ on $\Rp\setminus K$, and set $U_\varepsilon=1$ on $\Rp\setminus K$. The corresponding
Andersen--Kerman weight in the original variables is $w_\varepsilon(x)=x^{2\lambda+1-p}U_\varepsilon(x)$, so the exact
conjugation places it in $A_{p,\lambda}$ with $[w_\varepsilon]_{A_{p,\lambda}}=[U_\varepsilon]_{A_p(d\nu_\lambda)}$. For the commutator examples set $b_\varepsilon(a+\eta
t)=\log |t|$ on $K$ and extend $b_\varepsilon$ as the constant $0$ on $\Rp\setminus K$. The logarithmic singularity at $a$ is
the usual BMO singularity, and Lemma~\ref{lem:compact-localization} gives
$$
        [U_\varepsilon]_{A_p(d\nu_\lambda)} \approx [W_\varepsilon]_{A_p(-1,1)}=:A_\varepsilon,
        \qquad \|b_\varepsilon\|_{\BMO_\lambda}\approx_{\lambda,K} 1,
$$
and the $L^p$ norms in the variables $x$ and $t$ are comparable, with constants independent of $\varepsilon$.

Define the signed kernels
\begin{align}
        L_-(s,u) &=\varepsilon_\lambda\eta\, \mathcal K_\lambda(a+\eta s,a-\eta u)(a-\eta u)^{2\lambda+1}, \notag \\
        L_+(s,u) &=-\varepsilon_\lambda\eta\, \mathcal K_\lambda(a-\eta s,a+\eta u)(a+\eta u)^{2\lambda+1}. \notag
\end{align}
For $0<s<1/16$ and $0<u<1$, \eqref{eq:AK-local-sign-conj} and the local size estimate \eqref{eq:AK-local-conj-size} imply
\begin{align}\label{eq:sharpness-L-comparable}
        \frac{c}{s+u}\le L_\pm(s,u)\le\frac{C}{s+u},
\end{align}
with constants independent of $\varepsilon$.

Assume first that $1<p\le2$. For $x=a+\eta s$, $0<s<1/16$, the function $F_\varepsilon^{K}$ is supported on the other side of $a$,
namely $y=a-\eta u$, $0<u<1$. By the definition of $L_-(s,u)$, $$
        |\mathcal R_\lambda F_\varepsilon^{K}(a+\eta s)| 
        =\bigg|\int_0^1\mathcal K_\lambda(a+\eta s,a-\eta u)F_\varepsilon^{K}(a-\eta u)\cdot (a-\eta u)^{2\lambda+1}\eta\, \,du\bigg|
        \ge c\int_0^1 L_-(s,u)u^{-1+\varepsilon}\,du.
$$
If $p>2$, then $F_\varepsilon^{K}$ is supported on the right of $a$. Similarly, for $x=a-\eta s$, $0<s<1/16$, 
$$
        |\mathcal R_\lambda F_\varepsilon^{K}(a-\eta s)| \ge c\int_0^1 L_+(s,u)\,du.
$$

Consequently, for $1<p\le2$,
\begin{align*}
        \|\mathcal R_\lambda F_\varepsilon^{K}\|_{L^p(\Rp,U_\varepsilon\,d\nu_\lambda)}^p
        &\ge\int_0^\infty
        |\mathcal R_\lambda F_\varepsilon^{K}(a+\eta s)|^p
        U_\varepsilon(a+\eta s)(a+\eta s)^{2\lambda+1}\eta\,ds\\
        &\gtrsim
        \int_0^{\frac1{16}}
        |T_{\varepsilon,L_-}^-(s)|^p
        W_\varepsilon(s)(a+\eta s)^{2\lambda+1}\eta\,ds.
\end{align*}

Using \eqref{eq:sharpness-L-comparable}, Lemma~\ref{lem:one-sided-model} and compact localization further give
$$
\|\mathcal R_\lambda F_\varepsilon^{K}\|_{L^p(\Rp,U_\varepsilon\,d\nu_\lambda)}
\gtrsim
\|T_{\varepsilon,L_-}^-\|_{L^p((0,1/16),W_\varepsilon(s)\,ds)}
\gtrsim
A_\varepsilon^{\gamma_p}
\|F_\varepsilon\|_{L^p((-1,1),W_\varepsilon(t)\,dt)}
\gtrsim
A_\varepsilon^{\gamma_p}
\|F_\varepsilon^{K}\|_{L^p(\Rp,U_\varepsilon\,d\nu_\lambda)}.
$$
with constants independent of $\varepsilon$.
For $p>2$, the same computation on the left of $a$ gives the corresponding estimate with
$L_+$ and $T_{\varepsilon,L_+}^+$.

If the hypothetical bound for $R_\lambda$, equivalently for $\mathcal R_\lambda$, held with some
$\theta<\gamma_p$, then the last inequality would imply $A_\varepsilon^{\gamma_p}\le C A_\varepsilon^\theta$
as $A_\varepsilon\to\infty$, a contradiction. Hence $\theta\ge\gamma_p$.

The commutator lower bound uses the same separated supports. On these supports
$$
        b_\varepsilon(a+\eta s)-b_\varepsilon(a-\eta u)=\log s-\log u
$$
in the case $1<p\le2$, and 
$$
        b_\varepsilon(a-\eta s)-b_\varepsilon(a+\eta u)=\log s-\log u
$$
in the case $p>2$. Hence the commutator is represented by
the same kernels $L_\pm(s,u)$, with the extra factor $\log s-\log u$. 
Thus, when $1<p\le2$, for $0<s<1/16$ a similar argument gives
$$
\begin{aligned}
&\big|[b_\varepsilon,\mathcal R_\lambda]
F_\varepsilon^{K}(a+\eta s)\big|\gtrsim
\left|\int_0^1(\log s-\log u)L_-(s,u)u^{-1+\varepsilon}\,du\right|
=C_{\varepsilon,L_-}^-(s).
\end{aligned}
$$
For $p>2$, the analogous estimate holds at $a-\eta s$, with
$L_+$ and $C_{\varepsilon,L_+}^+(s)$.
Hence, for all sufficiently small $\varepsilon$, Lemma~\ref{lem:one-sided-model} and compact localization give
$$
        \|[b_\varepsilon,\mathcal R_\lambda]F_\varepsilon^{K}\|_{L^p(\Rp,U_\varepsilon\,d\nu_\lambda)} 
       \gtrsim A_\varepsilon^{2\gamma_p}
        \|F_\varepsilon\|_{L^p((-1,1),W_\varepsilon(t)\,dt)}
        \gtrsim \|b_\varepsilon\|_{\BMO_\lambda}
        A_\varepsilon^{2\gamma_p} \|F_\varepsilon^{K}\|_{L^p(\Rp,U_\varepsilon\,d\nu_\lambda)}.
$$
This gives the required commutator lower bound in both ranges.

A hypothetical commutator estimate with $\delta<2\gamma_p$ would force $A_\varepsilon^{2\gamma_p}\le C A_\varepsilon^\delta$ as
$A_\varepsilon\to\infty$, again a contradiction. Thus $\delta\ge2\gamma_p$.
\end{proof}

\bigskip
\noindent\textbf{Acknowledgements.} The author is deeply grateful to Professor Jill Pipher for many invaluable discussions and, in particular, for emphasizing the need to state explicitly the inclusion relations among the weight classes considered in Section 7. Her insight was essential in sharpening the statement and substantially improving the accuracy of the paper. The author would also like to  thank Liangchuan Wu for helpful discussions.


\medskip

\begin{thebibliography}{99}

\bibitem{AK81}
K. F. Andersen and R. A. Kerman, \emph{Weighted norm inequalities for generalized Hankel conjugate
transformations}, Studia Math. \textbf{71} (1981/82), no. 1, 15--26; doi: 10.4064/sm-71-1-15-26.

\bibitem{bcfr}
J. J. Betancor, A. Chicco Ruiz, J. C. Fari\~{n}a and L. Rodr\'iguez-Mesa, \emph{Maximal operators, Riesz transforms and Littlewood-Paley
functions associated with Bessel operators on BMO}, J. Math. Anal. Appl. \textbf{363} (2010), no. 1, 310--326.

\bibitem{bdt}
J. J. Betancor, J. Dziuba\'nski and J. L. Torrea, \emph{On Hardy spaces associated with
Bessel operators}, J. Anal. Math. \textbf{107} (2009), 195--219.

\bibitem{bfbmt}
J. J. Betancor, J. C. Fari\~{n}a, D. Buraczewski, T. Mart\'{\i}nez and J. L. Torrea, \emph{Riesz transform related to Bessel
operators}, Proc. Roy. Soc. Edinburgh Sect. A \textbf{137} (2007), no. 4, 701--725.

\bibitem{bfs}
J. J. Betancor, J. C. Fari\~{n}a and A. Sanabria, \emph{On Littlewood-Paley functions associated
with Bessel operators}, Glasg. Math. J. \textbf{51} (2009), no. 1, 55--70.

\bibitem{bhnv}
J. J. Betancor, E. Harboure, A. Nowak and B. Viviani, \emph{Mapping properties of fundamental operators in harmonic analysis
related to Bessel operators}, Studia Math. \textbf{197} (2010), no. 2, 101--140.

\bibitem{CS14}
A. J. Castro and T. Z. Szarek, \emph{Calder\'on--Zygmund operators in the Bessel setting for all possible type indices},
Acta Math. Sin. (Engl. Ser.) \textbf{30} (2014), no. 4, 637--648;
doi: 10.1007/s10114-014-2326-1.

\bibitem{CPP12}
D. Chung, M. C. Pereyra and C. P\'erez, \emph{Sharp bounds for general commutators on weighted
Lebesgue spaces}, Trans. Amer. Math. Soc. \textbf{364} (2012), no. 3, 1163--1177.

\bibitem{CRW76}
R. R. Coifman, R. Rochberg and G. Weiss, \emph{Factorization theorems for Hardy spaces in
several variables}, Ann.\ of Math. (2) \textbf{103} (1976), no. 3, 611--635.

\bibitem{CW77}
R. R. Coifman and G. Weiss, \emph{Extensions of Hardy spaces and their use in analysis}, Bull. Amer. Math. Soc.
\textbf{83} (1977), no. 4,  569--645.

\bibitem{DLWY}
X. T. Duong, J. Li, B. D. Wick and D. Yang, \emph{Factorization for Hardy spaces and characterization for BMO spaces via commutators in the
Bessel setting}, Indiana Univ. Math. J. \textbf{66} (2017), no. 4, 1081--1106.

\bibitem{DLMWY}
X. T. Duong, J. Li, S. Mao, H. Wu and D. Yang, \emph{Compactness of Riesz transform commutator associated
with Bessel operators}, J. Anal. Math. \textbf{135} (2018), no. 2, 639--673.

\bibitem{Ho}
K. P. Ho, \emph{Characterizations of BMO by $A_p$ weights and p-convexity}, Hiroshima
Math. J. {\bf41} (2011), no.2, 153--165.

\bibitem{HuntMuckenhouptWheeden1973}
R. Hunt, B. Muckenhoupt and R. Wheeden, \emph{Weighted norm inequalities for the conjugate function and
Hilbert transform}, Trans. Amer. Math. Soc. \textbf{176} (1973), 227--251.

\bibitem{k78}
R. A. Kerman, \emph{Boundedness criteria for generalized Hankel conjugate transformations}, Canad. J. Math.
\textbf{30} (1978), no. 1, 147--153.

\bibitem{Ler13}
A. K. Lerner, \emph{A simple proof of the $A_2$ conjecture}, Int. Math. Res. Not. IMRN \textbf{2013}, no. 14, 3159--3170.

\bibitem{Lorist21}
E. Lorist, \emph{On pointwise $\ell^r$-sparse domination in a space of homogeneous type}, J. Geom. Anal. \textbf{31} (2021), no. 9, 9366--9405.


\bibitem{LLLS}
J. Li, C.-W. Liang, F. Y.-H. Lin and C.-Y. Shen, \emph{Muckenhoupt-type weights in Bessel setting}, J. Geom. Anal.
\textbf{34} (2024), no. 7, Paper No. 192, 36 pp.

\bibitem{LLSW}
J. Li, C.-W. Liang, C.-Y. Shen and B. D. Wick, \emph{Muckenhoupt-type weights and quantitative weighted estimates in the Bessel setting},
Math. Z. \textbf{309} (2025), no. 1, Paper No. 13, 29 pp.

\bibitem{Li17}
K. Li, \emph{Two weight inequalities for bilinear forms},
Collect. Math. \textbf{68} (2017), no. 1, 129--144.

\bibitem{Muckenhoupt1972trans}
B. Muckenhoupt, \emph{Weighted norm inequalities for the Hardy maximal function}, Trans. Amer. Math. Soc. \textbf{165} (1972), 207--226.

\bibitem{MS65}
B. Muckenhoupt and E. M. Stein, \emph{Classical expansions and their relation to conjugate
harmonic functions}, Trans. Amer. Math. Soc. \textbf{118} (1965), 17--92.

\bibitem{Pet07}
S. Petermichl, \emph{The sharp bound for the Hilbert transform on weighted Lebesgue spaces in terms of the classical $A_p$
characteristic}, Amer. J. Math. \textbf{129} (2007), no. 5, 1355--1375.

\bibitem{v08}
M. Villani, \emph{Riesz transforms associated to Bessel operators}, Illinois J. Math. \textbf{52} (2008), no. 1, 77--89.

\end{thebibliography}
\end{document}